\documentclass[11pt,reqno]{amsart}

\usepackage[T1]{fontenc}
\usepackage[utf8]{inputenc}
\usepackage{amsmath,amssymb,amsfonts,amsthm}
\usepackage{enumerate}
\usepackage{microtype}
\usepackage{tikz}
\usetikzlibrary{arrows.meta}
\usepackage[hidelinks]{hyperref}

\numberwithin{equation}{section}
\allowdisplaybreaks

\newtheorem{theorem}{Theorem}[section]
\newtheorem{lemma}[theorem]{Lemma}
\newtheorem{corollary}[theorem]{Corollary}
\newtheorem{proposition}[theorem]{Proposition}

\theoremstyle{definition}
\newtheorem*{unnumbereddefinition}{Definition}

\theoremstyle{remark}
\newtheorem{remark}[theorem]{Remark}
\newtheorem*{unnumberedremark}{Remark}

\title[Instability of nonrotating stars]{On the nonlinear instability of nonrotating Stars}
\author{Zhiwu Lin}
\address{School of Mathematical Sciences, Fudan University, Shanghai, China}
\thanks{Z. Lin's research is supported in part by the National Natural Science Foundation of China under Grant 12494544.}
\email{zwlin@fudan.edu.cn}

\author{Xinlin Wu}
\address{School of Mathematics, Georgia Institute of Technology, Atlanta, Georgia, USA}
\email{xwu475@gatech.edu}

\date{}

\begin{document}

\begin{abstract}
We study radial nonlinear instability of compactly supported nonrotating equilibria of the three-dimensional Euler--Poisson system with a physical-vacuum boundary. Let $n^u(\mu)$ denote the radial instability index furnished by the turning-point theory of Lin and Zeng. Under general structural assumptions on the pressure law, suppose that $n^u(\mu)>0$ and that the equilibrium is not a mass extremum, $M'(\mu)\neq0$. Conditional on the existence of a sufficiently regular radial solution on the relevant time interval, we establish two nonlinear escape criteria. First, every perturbation with Hamiltonian strictly below that of the equilibrium exits a fixed neighborhood on a logarithmic time scale controlled by the least unstable linear growth rate. For the subclass of data for which the associated Lyapunov functional is initially nonnegative, we also obtain an explicit exponential lower bound in the weighted displacement norm. Second, sufficiently small perturbations whose Riesz projection onto the finite-dimensional unstable subspace is not too small escape on a logarithmic time scale. The second argument uses the exponential trichotomy of the linearized Hamiltonian flow and an invariant-cone estimate. In the polytropic class, the conditional estimates combine with the radial physical-vacuum local theory to yield unconditional nonlinear instability in the corresponding classical-solution topology. The results complement Jang's nonlinear instability theorem for Lane--Emden stars by treating mechanisms that do not require initial alignment with a leading growing eigenmode and by applying, conditionally, to unstable branches for general equations of state.
\end{abstract}


\maketitle
\section{Introduction}\label{intro}

\subsection{The model, equilibria, and the turning-point index}
We consider a self-gravitating barotropic gas governed in Eulerian coordinates by the three-dimensional Euler--Poisson system
\begin{align}
  \rho_t+\nabla\cdot(\rho u)&=0, \label{continuity eq}\\
  \rho D_tu&=-\nabla P(\rho)-\rho\nabla V, \label{euler eq}\\
  \Delta V&=4\pi\rho,\qquad \lim_{|x|\to\infty}V(t,x)=0. \label{eq for potential}
\end{align}
Here $\rho\ge0$ is the density, $u\in\mathbb R^3$ is the velocity, $D_t=\partial_t+u\cdot\nabla$ is the material derivative, and $V$ is the self-consistent gravitational potential.  We assume
\begin{equation}\label{property of pressure 1}
 P\in C^1((0,\infty)),\qquad P'(s)>0,
\end{equation}
and normalize the pressure by $\lim_{s\to0+}P(s)=P(0)=0$; this limit exists under the asymptotics below, and subtracting a constant does not change the equations.  We further assume that there are $\gamma_0\in(6/5,2)$ and $C_{\gamma_0}>0$ such that
\begin{equation}\label{property of pressure 2}
 \lim_{s\to0+}s^{1-\gamma_0}P'(s)=C_{\gamma_0}.
\end{equation}
Thus the square of the sound speed has the physical-vacuum asymptotics $P'(\rho)\asymp\rho^{\gamma_0-1}$ near vacuum.  We introduce the internal-energy potential $\Phi$ by
\begin{equation}\label{enthalpy}
 \Phi(0)=\Phi'(0)=0,\qquad \Phi''(s)=\frac{P'(s)}{s};
\end{equation}
its derivative $\Phi'$ is the specific enthalpy.  See \cite{subrahmanijanchandrasekhar_1967_an} for the astrophysical background.

A \emph{nonrotating star} is a compactly supported static solution $(\rho_\mu,0)$ of \eqref{continuity eq}--\eqref{eq for potential}.  Its density is radial by \cite{Gidas1981SymmetryOP}; we write $\mu=\rho_\mu(0)$ for the central density, $S_\mu=B(0,R_\mu)$ for its support, and
\[
 M(\mu)=\int_{\mathbb R^3}\rho_\mu(x)\,dx
\]
for its total mass.  The following facts are recalled from \cite{separable}.

\begin{lemma}[Existence and physical-vacuum behavior]\label{(Existence of nonrotating star)}
Under \eqref{property of pressure 1}--\eqref{property of pressure 2}, there exists $\mu_0>0$ such that, for every $\mu\in(0,\mu_0)$, there is a compactly supported radial equilibrium $(\rho_\mu,0)$ with support $S_\mu=B(0,R_\mu)$ and $\rho_\mu(0)=\mu$.  The profile depends $C^1$-smoothly on $(r,\mu)$.  Moreover, for $r$ sufficiently close to $R_\mu$,
\begin{equation}\label{two-sided-vacuum-asymptotics}
 \rho_\mu(r)\asymp (R_\mu-r)^{\frac1{\gamma_0-1}},\qquad
 |\rho_\mu'(r)|\asymp (R_\mu-r)^{\frac{2-\gamma_0}{\gamma_0-1}},
\end{equation}
and consequently
\begin{equation}\label{two-sided-Phi-asymptotics}
 \Phi''(\rho_\mu(r))\asymp (R_\mu-r)^{\frac{\gamma_0-2}{\gamma_0-1}}.
\end{equation}
\end{lemma}

\begin{remark}
For $P(\rho)=C_\gamma\rho^\gamma$, $6/5<\gamma<2$, the equilibria are the Lane--Emden stars and exist for every $\mu>0$; see, for example, \cite{subrahmanijanchandrasekhar_1967_an,heinzle_2003_newtonian}.
\end{remark}

The radial linear stability of $(\rho_\mu,0)$ is described by the turning-point principle of Lin and Zeng.  We retain the notation $D_\mu^0$, $L_\mu$, and $i_\mu$ from \cite{separable}; the operators are recalled in Subsection~\ref{notation-subsection}.

\begin{theorem}[Turning-point principle, \cite{separable}]\label{turning pt principle}
The equilibrium $\rho_\mu$ is spectrally stable under radial perturbations if and only if
\[
 n^-(D_\mu^0)=1\qquad\text{and}\qquad i_\mu=1,
\]
where
\[
 i_\mu=
 \begin{cases}
 1,& M'(\mu)\,\dfrac{d}{d\mu}\!\left(\dfrac{M(\mu)}{R_\mu}\right)>0\ \text{or}\ M'(\mu)=0,\\[1.2ex]
 0,& M'(\mu)\,\dfrac{d}{d\mu}\!\left(\dfrac{M(\mu)}{R_\mu}\right)<0\ \text{or}\ \dfrac{d}{d\mu}\!\left(\dfrac{M(\mu)}{R_\mu}\right)=0.
 \end{cases}
\]
The quantities $M'(\mu)$ and $\frac{d}{d\mu}(M(\mu)/R_\mu)$ do not vanish simultaneously.  Defining the radial instability index by
\[
 n^u(\mu):=n^-(D_\mu^0)-i_\mu=n^-(L_\mu)-i_\mu\ge0,
\]
one has spectral stability if and only if $n^u(\mu)=0$.  The index can change only at extrema of the mass curve.  In particular, for sufficiently small $\mu$,
\[
 n^u(\mu)=
 \begin{cases}
 1,&6/5<\gamma_0<4/3,\\
 0,&4/3<\gamma_0<2.
 \end{cases}
\]
\end{theorem}

\subsection{Context and main results}
For Lane--Emden stars, the classical threshold $\gamma=4/3$ separates the stable and unstable small-density regimes.  Radial linear stability and instability were established by Lin \cite{lin_1997_stability}.  Jang proved fully nonlinear instability for $\gamma=6/5$ in \cite{jang_2008_nonlinear} and for $6/5<\gamma<4/3$ in \cite{Jang_2013_nonlinear}. More recently, Cheng, Cheng, and Lin constructed expanding solutions starting arbitrarily close to unstable Lane--Emden stars for $6/5<\gamma<4/3$ \cite{Expanding_solution}.  In contrast, we are not aware of a rigorous construction of collapsing
solutions from initial data arbitrarily close to an unstable Lane--Emden star in the range $6/5<\gamma<4/3$.  Existing rigorous constructions of gravitational collapse occur in rather different regimes.  In particular, Guo, Had\v{z}i\'c, Jang, and Schrecker constructed exactly self-similar collapsing solutions for $1<\gamma<4/3$ \cite{Collapse_solution}; these solutions are not perturbations of Lane--Emden equilibria and, unlike the compactly supported Lane--Emden stars, have infinite total mass.  For general equations of state, the turning-point theory \cite{separable} shows that unstable branches can also occur outside the classical Lane--Emden picture and can carry more than one unstable radial mode.  On the stable branch, nonlinear radial stability under general pressure assumptions was obtained in \cite{lin_2024_nonlinear}.  The natural complementary question is whether the unstable index $n^u(\mu)>0$ forces nonlinear escape for general pressure laws.

The present paper gives two conditional nonlinear instability mechanisms.  Throughout the statements, ``conditional'' means that a sufficiently regular radial solution is assumed to exist on the logarithmic time interval appearing in the conclusion.  This distinction is necessary because the available three-dimensional physical-vacuum well-posedness theorems with self-gravity are stated for polytropic pressure laws; see the discussion following Theorem~\ref{Main result when energy is higher}.

\subsubsection*{Lower-energy instability}
Assume
\begin{equation}\label{main-index-assumptions-intro}
 n^u(\mu)>0,\qquad M'(\mu)\ne0.
\end{equation}
The \emph{displacement Hessian} is the second variation of the Hamiltonian with respect to the Lagrangian map $X$ at the equilibrium $(\mathrm{id},0)$.  Equivalently, it is the quadratic form governing the leading change of the potential energy under infinitesimal mass-preserving radial displacements.  The condition $M'(\mu)\ne0$ makes this Hessian nondegenerate: it excludes the zero mode that occurs at a mass turning point.  Such turning points are precisely the possible transition points at which the number of unstable modes can change in the turning-point theory of Lin and Zeng \cite{separable}.  The condition $n^u(\mu)>0$ then guarantees that the displacement Hessian has a negative direction on the dynamically accessible radial class.

Our first result, Theorem~\ref{Main thm}, shows that any sufficiently small radial solution with
\[
 H(X(0),U(0))<H(\mathrm{id},0)
\]
must leave a neighborhood of the equilibrium within a time bounded by
\[
 \frac{1+o(1)}{\sqrt{\mu_1}}
 \log\!\left(\frac{C\delta}{|H(X(0),U(0))|}\right),
\]
where $-\mu_1<0$ is the negative eigenvalue closest to the origin.  Thus the least unstable linear rate controls the longest possible escape time.  The argument does not require the initial perturbation to have a prescribed unstable spectral component. It uses conservation of the Hamiltonian together with a Lyapunov functional built from all negative eigendirections of the displacement Hessian, that is, the radial directions along which the quadratic part of the potential energy decreases. To describe the mechanism more precisely, we use the Lagrangian
variables introduced in Section~2.  Let \(X(t)\) be the Lagrangian
flow map and let \(U(t):=\rho_\mu X_t(t)\) be its conjugate momentum,
so that the equilibrium star corresponds to
\((X,U)=(\mathrm{id},0)\).  If \(w_i\) are the negative eigenfunctions
of the displacement Hessian, define the displacement and momentum
coordinates in these directions by
\[
    a_i(t)=[X(t)-\mathrm{id},w_i],
    \qquad
    b_i(t)=\langle U(t),w_i\rangle_{L^2}.
\]
The Lyapunov functional
\[
    V(t)=\sum_i a_i(t)b_i(t)
\]
then measures the coupling between displacement and momentum along
all hyperbolic directions. In the one-mode picture, $V$ is the product of the displacement and velocity coordinates; its sign distinguishes data already moving toward a growing branch from data initially dominated by the decaying branch.  For the subclass of data satisfying \(V(0)\ge 0\), the theorem gives
an explicit exponential lower bound for
\(\|X(t)-\mathrm{id}\|_{Y_\mu}^2\) starting at \(t=0\).
When \(V(0)<0\), the solution may first undergo a stable-dominated
transient.  The proof bounds the time needed either to leave the
neighborhood or to enter the growing regime, but it does not identify
a distinguished transition time from which an exponential lower bound
of the form above can be stated.  Accordingly, no such explicit lower
bound is asserted in this case.

\subsubsection*{Instability from an unstable spectral component}
Theorem~\ref{Main result when energy is higher} treats data of either sign of the Hamiltonian.  Let $P_u$ be the Riesz projection onto the finite-dimensional unstable subspace of the linearized Hamiltonian flow.  If
\[
 \|P_uZ(0)\|\ge \kappa\|Z(0)\|,
 \qquad Z=(X-\mathrm{id},U),
\]
for a fixed cone aperture $\kappa\in(0,1]$, then the solution exits a small neighborhood on a logarithmic time scale.  The admissible size of the full perturbation is allowed to depend on $\kappa$ and decreases as the initial data approach the center-stable subspace.  The proof uses the exponential trichotomy of the Hamiltonian group and an adapted norm on its unstable, center, and stable subspaces.  An invariant-cone estimate shows that the unstable component grows at an effective rate converging to $\sqrt{\mu_1}$ as the perturbation radius tends to zero.

\subsubsection*{From conditional to fully nonlinear instability}
For $P(\rho)=C_\gamma\rho^\gamma$, $1<\gamma<3$, Gu and Lei \cite{GU2016662} proved local well-posedness of the three-dimensional Euler--Poisson physical-vacuum problem in a high-order weighted Sobolev class.  For the radial ball geometry, one may instead use the local theory of Luo, Xin, and Zeng \cite{Xinzhouping}.  These solution classes control the $W^{1,\infty}$ deformation norm required by the conditional arguments.  A stopping-time continuation argument therefore converts the conditional estimates into fully nonlinear radial instability in the polytropic case; see Theorem~\ref{Main thm 0}.  For a general pressure law satisfying only \eqref{property of pressure 1}--\eqref{property of pressure 2}, the leading physical-vacuum degeneracy is still determined by the sound speed $P'(\rho)$, and an extension of the polytropic local theory is plausible under substantially stronger smoothness, nondegeneracy, and compatibility assumptions.  Such an extension requires a separate construction and is not used here; we have not found a published three-dimensional Euler--Poisson theorem in precisely the generality needed for the present application.

\subsubsection*{Comparison with Jang}
Jang's theorem \cite{Jang_2013_nonlinear} constructs, for every sufficiently small amplitude $\delta$, a Lane--Emden perturbation of size $O(\delta)$ in a high-order weighted physical-vacuum energy space, seeded along the fastest growing linear mode.  Jang proves that at the sharp logarithmic time
\[
 T^\delta=\frac1{\sqrt{\mu_0}}\log\frac{2\theta_0}{\delta},
\]
the solution has reached a fixed size in the weaker zeroth-order instantaneous energy $E^0$.  Jang's theorem is therefore a \emph{strong-to-weak} nonlinear instability result: the high-order weighted energy controls the solution and closes the nonlinear bootstrap, while escape is detected in a lower-order weighted norm.  Since this weaker norm is continuously controlled by the high-order norm, escape in the weaker topology automatically implies escape in the stronger one; in this precise topological sense, Jang's conclusion is stronger than a strong-to-strong instability statement.

The two results nevertheless address different aspects of instability.  Jang's theorem has the sharp leading-mode time scale, and gives the stronger strong-to-weak escape conclusion, but it is proved for a specially prepared family of data aligned with the fastest growing mode.  The present turning-point formulation applies conditionally to unstable branches for general equations of state; its lower-energy criterion does not require initial alignment with a growing eigenfunction, and its unstable-projection theorem allows arbitrary combinations of unstable modes and initial energies of either sign.  Our unconditional polytropic conclusion in Theorem~\ref{Main thm 0} is strong-to-strong in the chosen local well-posedness topology and is therefore weaker topologically than Jang's conclusion, while applying to broader classes of perturbations.  The methods are correspondingly different: Jang uses a hierarchy of weighted high-order energies and a Duhamel comparison with the fastest growing mode, whereas our lower-energy theorem exploits the separable Hamiltonian structure and a finite-dimensional Lyapunov functional.

\subsubsection*{Key analytical ingredients}
The proof combines four ingredients.  First, the Lagrangian formulation builds the conservation of mass into the choice of variables and identifies $(\mathrm{id},0)$ as a critical point of a separable Hamiltonian.  Its second variation with respect to the Lagrangian displacement is the self-adjoint operator $\widetilde{\mathbb L}_\mu$ on the weighted displacement space.  This identifies the negative spectral directions of the potential energy with the hyperbolic directions of the linearized Hamiltonian flow.

Second, the turning-point index and the nondegeneracy condition $M'(\mu)\ne0$ imply that the negative subspace is finite dimensional, that zero is separated from the spectrum, and that the positive and negative parts of the absolute quadratic form control the natural weighted form norm.  These facts yield both the finite-dimensional Lyapunov functional used in the lower-energy theorem and the Riesz projections used in the unstable-cone theorem.

Third, the nonlinear force must be compared with its linearization in a topology compatible with the physical vacuum.  Lemma~\ref{Estimation of dH/dX-Lmu} gives a modulus-of-continuity remainder estimate in a small $W^{1,\infty}$ neighborhood.  Radial symmetry is essential here: the internal force admits an exact Piola-stress formula, and Newton's shell theorem reduces the gravitational energy to a one-dimensional mass-shell integral.  These formulas avoid the derivative loss and singular convolution estimates that arise in a direct three-dimensional expansion.

Finally, the two nonlinear mechanisms use different aspects of the Hamiltonian structure.  For lower-energy data, conservation of the Hamiltonian controls the positive spectral part by the negative part and forces a finite-dimensional Lyapunov functional to grow.  For data with a prescribed unstable component, an exponential trichotomy and an invariant-cone argument control the stable and center components while preserving the asymptotically sharp growth rate.  In the polytropic class, a centered high-order distance associated with the radial physical-vacuum local theory supplies the continuation topology needed to iterate the local solution up to the logarithmic escape time.

\subsection{Organization of the paper}
Section~\ref{Lagrangian Formulation and Hamiltonian Structure} derives the Lagrangian Hamiltonian formulation and computes its first two variations.  Section~\ref{section of criterion} isolates the abstract lower-energy mechanism.  The main conditional and polytropic results are stated in Section~4.  Section~5 contains the weighted Hardy, potential, and nonlinear-remainder estimates.  Section~6 proves the spectral properties and the two instability theorems; the unstable-projection argument is based on the exponential trichotomy.  Section~7 proves Lemma~\ref{Estimation of dH/dX-Lmu}.

\subsection{Notation}\label{notation-subsection}
We collect the weighted spaces and operators used below.  For a positive
weight $q$, $L_q^2(S_\mu)$ denotes the space with norm
$\|f\|_{L_q^2}^2=\int_{S_\mu}q|f|^2\,dx$.  Set
\begin{align*}
 X_\mu&:=L^2_{\Phi''(\rho_\mu)}(S_\mu),
 &Y_\mu&:=\bigl(L^2_{\rho_\mu}(S_\mu)\bigr)^3,\\
 Z_\mu&:=\left\{X\in Y_\mu:
       \operatorname{div}(\rho_\mu X)\in X_\mu\right\},
 &\vec Z_\mu&:=Z_\mu\times Y_\mu^*.
\end{align*}
The corresponding inner products are
\begin{align*}
 \langle\sigma_1,\sigma_2\rangle_{X_\mu}
 &=\int_{S_\mu}\Phi''(\rho_\mu)\sigma_1\sigma_2\,dx,\\
 [X_1,X_2]:=\langle X_1,X_2\rangle_{Y_\mu}
 &=\int_{S_\mu}\rho_\mu X_1\cdot X_2\,dx,\\
 \langle X_1,X_2\rangle_{Z_\mu}
 &=[X_1,X_2]
 +\int_{S_\mu}\Phi''(\rho_\mu)
   \operatorname{div}(\rho_\mu X_1)
   \operatorname{div}(\rho_\mu X_2)\,dx.
\end{align*}
The radial subspaces are
\begin{align*}
 X_{\mu,r}&:=\{\sigma\in X_\mu:\sigma(x)=\sigma(|x|)\},\\
 Y_{\mu,r}&:=\left\{X\in Y_\mu:
     X(x)=v(|x|)\frac{x}{|x|}\ \text{for some }v\right\},\\
 Z_{\mu,r}&:=\left\{X\in Y_{\mu,r}:
     \operatorname{div}(\rho_\mu X)\in X_{\mu,r}\right\},
 \qquad
 \vec Z_{\mu,r}:=Z_{\mu,r}\times Y_{\mu,r}^*.
\end{align*}
All perturbations in this paper are radial.  After restricting the spaces and
operators below to their radial subspaces, we omit the subscript $r$ unless it
is needed for clarity.

The basic operators are
\begin{align*}
 L_\mu&:=\Phi''(\rho_\mu)-4\pi(-\Delta)^{-1}:X_\mu\to X_\mu^*,\\
 I_{X_\mu}&:=\frac1{\Phi''(\rho_\mu)}:X_\mu^*\to X_\mu,\\
 \mathbb L_\mu&:=I_{X_\mu}L_\mu
 =I-\frac{4\pi}{\Phi''(\rho_\mu)}(-\Delta)^{-1},\\
 D_\mu&:=-\Delta-\frac{4\pi}{\Phi''(\rho_\mu)},
 \qquad D_\mu^0:=D_\mu\big|_{\mathrm{radial}},\\
 A_\mu&:=\rho_\mu:Y_\mu\to Y_\mu^*,\\
 B_\mu&:=-\operatorname{div}:Y_\mu^*\supset D(B_\mu)\to X_\mu,
 \qquad B_\mu':=\nabla:X_\mu^*\supset D(B_\mu')\to Y_\mu,\\
 \widetilde L_\mu&:=A_\mu B_\mu' L_\mu B_\mu A_\mu:Y_\mu\to Y_\mu^*,\\
 \widetilde{\mathbb L}_\mu&:=A_\mu^{-1}\widetilde L_\mu
 =B_\mu'L_\mu B_\mu A_\mu:Y_\mu\to Y_\mu.
\end{align*}
For $\sigma\in X_\mu$, the Newtonian inverse is understood after extension by
zero outside $S_\mu$:
\begin{equation}\label{laplacian inverse}
 (-\Delta)^{-1}\sigma(x)
 =\frac1{4\pi}\int_{S_\mu}\frac{\sigma(y)}{|x-y|}\,dy.
\end{equation}

If $f:X\to X^*$ is a self-dual operator on a Hilbert space $X$, its negative
index is
\[
 n^-(f):=\max\left\{\dim S:S\subset D(f),\ 
          \langle fu,u\rangle<0\ \text{for every }0\ne u\in S\right\}.
\]
When the Riesz representative of $f$ has a complete orthogonal eigenbasis,
$n^-(f)$ is the number of its negative eigenvalues, counted with
multiplicity.  We use $D$ and $\nabla$ for spatial derivatives and use
$\langle\cdot,\cdot\rangle$ without a subscript for the natural duality
pairing when no ambiguity can arise.

\section{Lagrangian Formulation and Hamiltonian Structure}\label{Lagrangian Formulation and Hamiltonian Structure}
The Lagrangian formulation and Hamiltonian structure are briefly discussed in the appendix of \cite{separable}. To make the present paper self-contained, we provide the derivation needed below. Fix a nonrotating star $(\rho_\mu(|x|),0)$ supported in the ball $S_\mu:=B(0,R_\mu)$, where $\mu=\rho_\mu(0)$ is the central density. We take $S_\mu$ as the reference domain and define the configuration space of Lagrangian maps by
\[
\Lambda:=\{\text{diffeomorphism }X: S_\mu \to X(S_\mu)\subseteq \mathbb{R}^3\}.
\]
Here, each Lagrangian map satisfies the following ODE.
\begin{equation}\label{Lagrangian flow}
\begin{aligned}
    \frac{d}{dt}X(x,t)&=u(X(x,t),t),\\
    X(x,0)&=x.
\end{aligned}
\end{equation}
For any reference density $\rho_*:S_\mu \to \mathbb{R}_{\geq 0}$ and a path of Lagrangian maps $X(t)\in \Lambda$, we define the action functional $\mathcal{A}$ by
\[
\mathcal{A}=\int  \left( \int_{S_\mu} \frac{1}{2} \rho_*|X_t|^2dx-\int_{X(S_\mu,t)} \Phi (\rho) dx+\frac{1}{8\pi} \int_{\mathbb{R}^3}|\nabla V|^2dx \right)dt,
\]
where the physical density $\rho$ in the Eulerian coordinate is given by 
\begin{equation}\label{def of density in Lag}
    \rho(x,t)=\left( \frac{\rho_*}{\det DX(x,t)}\right)\circ X(x,t)^{-1}: X(S_\mu,t)\to \mathbb{R}_{\geq0},
\end{equation}
which can be extended as 0 outside $X(S_\mu,t)\subseteq \mathbb{R}^3$. This definition ensures conservation of mass. The gravitational potential $V(x,t)$ is then defined via \eqref{eq for potential}. The internal-energy potential $\Phi$ satisfies \eqref{enthalpy}.
\begin{lemma}
 $X(t)$ is a critical path of $\mathcal{A}$ if and only if $(\rho,u=X_t\circ X^{-1})$ solves the Euler--Poisson system \eqref{continuity eq}--\eqref{eq for potential}.
\end{lemma}
\begin{proof}
The push-forward formula \eqref{def of density in Lag} gives
\[
 \rho\circ X=\frac{\rho_*}{J},\qquad J=\det DX.
\]
Differentiating in time and using Jacobi's formula yields
\[
 (\rho_t+u\cdot\nabla\rho+\rho\nabla\cdot u)\circ X=0,
 \qquad u=X_t\circ X^{-1},
\]
so the continuity equation holds identically.

Let $\widetilde X$ be a variation compactly supported in time and put
$w=\widetilde X\circ X^{-1}$.  After integration by parts in time, the
variation of the kinetic term is
\[
 -\iint_{X(S_\mu,t)}\rho D_tu\cdot w\,dx\,dt.
\]
The corresponding Eulerian density variation is
\begin{equation}\label{partial epsilon of rho}
 \dot\rho=-\nabla\cdot(\rho w).
\end{equation}
Since $P(\rho)=\rho\Phi'(\rho)-\Phi(\rho)$, integration by parts gives
\[
 \left.\frac d{d\epsilon}\right|_{\epsilon=0}
 \int_{X_\epsilon(S_\mu,t)}\Phi(\rho_\epsilon)\,dx
 =\int_{X(S_\mu,t)}\nabla P(\rho)\cdot w\,dx.
\]
Finally, using $\Delta V=4\pi\rho$, self-adjointness of the Newtonian
potential, and \eqref{partial epsilon of rho},
\[
 \left.\frac d{d\epsilon}\right|_{\epsilon=0}
 \frac1{8\pi}\int_{\mathbb R^3}|\nabla V_\epsilon|^2\,dx
 =-\int_{X(S_\mu,t)}\rho\nabla V\cdot w\,dx.
\]
Consequently,
\[
 \delta\mathcal A(X)[\widetilde X]
 =-\iint_{X(S_\mu,t)}
 \bigl(\rho D_tu+\nabla P(\rho)+\rho\nabla V\bigr)\cdot w\,dx\,dt.
\]
Because $\widetilde X$, and hence $w$, is arbitrary, critical paths are
exactly the solutions of the Euler--Poisson momentum equation.  Together
with the continuity equation and Poisson equation, this proves the lemma.
\end{proof}

We now choose the reference density to be the equilibrium, $\rho_*=\rho_\mu$,
and introduce the momentum variable
\[
 U:=\rho_\mu X_t\in Y_\mu^*
 =\left(L^2_{1/\rho_\mu}(S_\mu)\right)^3.
\]
The Hamiltonian is
\begin{equation}\label{Hamiltonian-definition}
 H(X,U)=\frac12\int_{S_\mu}\frac{|U|^2}{\rho_\mu}\,dx
 +\int_{X(S_\mu)}\Phi(\rho)\,dx
 -\frac1{8\pi}\int_{\mathbb R^3}|\nabla V|^2\,dx+c,
\end{equation}
where $c$ is chosen so that $H(\mathrm{id},0)=0$.  In the primal--dual
variables $(X,U)\in Y_\mu\times Y_\mu^*$, the canonical Hamiltonian operator
maps $(p,q)\in Y_\mu^*\times Y_\mu$ to $(q,-p)$.  Thus the Hamilton equations
are
\begin{equation}\label{Hamiltonian structure of E-P}
 \binom{X_t}{U_t}
 =\binom{\delta H/\delta U}{-\delta H/\delta X}
 =\binom{\rho_\mu^{-1}U}
 {-\rho_\mu\nabla\bigl(\Phi'(\rho)+V\bigr)\circ X}.
\end{equation}
The equilibrium $(\mathrm{id},0)$ is a critical point.  More explicitly,
\begin{equation}\label{first variation against X of H}
 \left\langle\frac{\delta H}{\delta X}(X),\widetilde X\right\rangle
 =\int_{S_\mu}
 \bigl(\nabla[\Phi'(\rho)+V]\circ X\bigr)
 \cdot(\rho_\mu\widetilde X)\,dx.
\end{equation}

The stability properties of the Hamiltonian system are closely related to the second variation of the Hamiltonian functional at the equilibrium  $(\mathrm{id},0)$, in particular its second variation with respect to the Lagrangian displacement.  Rein \cite{rein_2003_nonlinear} proved conditional nonlinear stability when the initial density is an energy minimizer, and Chandrasekhar's variational principle likewise links stability to constrained energy minimization; see \cite{aly_1992_on,Antonov_1987,Lebovitz}.  In the Lagrangian formulation the mass constraint is built into the push-forward relation \eqref{def of density in Lag}, and the kinetic second variation is nonnegative.  It therefore remains to analyze the second variation of the Hamiltonian with respect to the Lagrangian displacement.  We refer to this quadratic form as the \emph{displacement Hessian}; it measures the leading change of the potential energy under an infinitesimal mass-preserving deformation of the reference star.  The separable structure of the Hamiltonian motivates the lower-energy mechanism in Section~\ref{section of criterion}.
We next calculate the second variation of $H$ with respect to $X$ at the steady state $(\mathrm{id},0)$.
\begin{lemma}\label{second variation of H}
    We define the  potential part of $H$ by
    \[
    F_1(X)=\int_{X(S_\mu)}\Phi(\rho)dx-\frac{1}{8\pi}\int_{\mathbb{R}^3}|\nabla V|^2dx+c.
    \]
Then, for $X(\varepsilon)=\mathrm{id}+\varepsilon\widetilde X$, we have
\begin{align*}
 \left.\partial_\varepsilon^2F_1(X)\right|_{\varepsilon=0}
 &=\left\langle
 \left.\frac{\delta^2H}{\delta X^2}\right|_{\mathrm{id}}
 \widetilde X,\widetilde X\right\rangle\\
 &=\int_{S_\mu}\Phi''(\rho_\mu)\sigma^2\,dx
 -\frac1{4\pi}\int_{\mathbb R^3}
 \left|\nabla\bigl(|x|^{-1}*\sigma\bigr)\right|^2\,dx,
\end{align*}
where
$\sigma=\left.\partial_\varepsilon\rho\right|_{\varepsilon=0}
=-\nabla\cdot(\rho_\mu\widetilde X)$.
\end{lemma}

\begin{proof}
It suffices first to consider a smooth radial displacement $\widetilde X$;
the general statement follows by density in the form domain.  Let
$X_\varepsilon=\mathrm{id}+\varepsilon\widetilde X$, let
$\rho_\varepsilon$ be the corresponding push-forward density, and set
\[
 \sigma:=\left.\partial_\varepsilon\rho_\varepsilon\right|_{\varepsilon=0}
 =-\operatorname{div}(\rho_\mu\widetilde X).
\]
The standard transport identities for an affine Lagrangian variation give
\begin{equation}\label{second-density-transport-identity}
 \left.\partial_\varepsilon^2\rho_\varepsilon\right|_{\varepsilon=0}
 =\operatorname{div}\operatorname{div}
   (\rho_\mu\widetilde X\otimes\widetilde X).
\end{equation}
Applying these identities to the internal energy yields
\begin{align}
 \left.\frac{d^2}{d\varepsilon^2}\right|_{\varepsilon=0}
 \int_{X_\varepsilon(S_\mu)}\Phi(\rho_\varepsilon)\,dx
 &=\int_{S_\mu}\Phi''(\rho_\mu)\sigma^2\,dx\notag\\
 &\quad+
 \int_{S_\mu}\Phi'(\rho_\mu)
 \operatorname{div}\operatorname{div}
 (\rho_\mu\widetilde X\otimes\widetilde X)\,dx.
 \label{second-internal-variation-concise}
\end{align}
For the gravitational energy, differentiating Poisson's equation twice and
integrating by parts gives
\begin{align}
 \left.\frac{d^2}{d\varepsilon^2}\right|_{\varepsilon=0}
 \left[-\frac1{8\pi}\int_{\mathbb R^3}|\nabla V_\varepsilon|^2\,dx\right]
 &=\int_{S_\mu}V_\mu
 \operatorname{div}\operatorname{div}
 (\rho_\mu\widetilde X\otimes\widetilde X)\,dx\notag\\
 &\quad-\frac1{4\pi}\int_{\mathbb R^3}
 \left|\nabla(|x|^{-1}*\sigma)\right|^2\,dx.
 \label{second-gravity-variation-concise}
\end{align}
The equilibrium equation implies
$\nabla(\Phi'(\rho_\mu)+V_\mu)=0$ in $S_\mu$.  Hence
$\Phi'(\rho_\mu)+V_\mu$ is constant there.  The sum of the transport
terms in \eqref{second-internal-variation-concise} and
\eqref{second-gravity-variation-concise} is therefore this constant times
$\int_{\mathbb R^3}\partial_\varepsilon^2\rho_\varepsilon|_0\,dx$,
which vanishes by conservation of mass.  The remaining terms are exactly the
asserted quadratic form.
\end{proof}
\begin{remark}\label{rmk for second variation of H}
Polarization gives
\[
\left\langle \frac{\delta^2H}{\delta X^2}\Big\vert_{\mathrm{id}}\tilde{X}_1,\tilde{X}_2
\right\rangle=\int_{S_\mu} \Phi^{\prime\prime} (\rho_\mu) \sigma_1\sigma_2 dx -\frac{1}{4\pi} \int_{\mathbb{R}^3}  \nabla (|x|^{-1}*\sigma_1) \cdot \nabla (|x|^{-1}*\sigma_2) dx,
\]
where $\sigma_i=-\nabla\cdot (\rho_\mu \tilde{X}_i)$, $i=1,2$.
\end{remark}

The operators introduced in Subsection~\ref{notation-subsection} satisfy
\[
 \left\langle \frac{\delta^2H}{\delta X^2}\Big|_{\mathrm{id}}X,X\right\rangle
 =\langle L_\mu B_\mu A_\mu X,B_\mu A_\mu X\rangle
 =\langle\widetilde L_\mu X,X\rangle,
 \qquad X\in D(B_\mu A_\mu).
\]
Thus the displacement Hessian is represented by $\widetilde L_\mu$, or by
its self-adjoint Riesz representative $\widetilde{\mathbb L}_\mu$ on
$Y_\mu$.  From this point onward all spaces and operators are restricted to
the radial subspaces described in Subsection~\ref{notation-subsection}.

\section{A lower-energy mechanism for separable Hamiltonian systems}\label{section of criterion}
This section isolates the mechanism behind Theorem~\ref{Main thm}.  For simplicity, we formulate it at the level of a second-order Hamiltonian equation, with all analytic estimates included as hypotheses. This methodology can be generalized to separable Hamiltonian systems.  In Sections~5--6 the hypotheses are verified directly for the Euler--Poisson problem.

Let $\mathcal V\hookrightarrow\mathcal H$ be real Hilbert spaces, with dense continuous embedding, and let $\mathbb L$ be self-adjoint on $\mathcal H$ with form domain $\mathcal V$.  Assume that $0\notin\sigma(\mathbb L)$ and that the negative spectral subspace is finite dimensional.  We write
\[
 \mathcal H=\mathcal H_-\oplus\mathcal H_+,
 \qquad
 \mathbb Lw_i=-\mu_iw_i,
 \quad 0<\mu_1\le\cdots\le\mu_n,
\]
where $\{w_i\}_{i=1}^n$ is orthonormal, and assume that the positive form satisfies
\[
 [\mathbb Lv,v]\ge \lambda_1\|v\|_{\mathcal H}^2,
 \qquad v\in\mathcal H_+,
\]
for some $\lambda_1>0$.  We assume, in addition, that the form norm is equivalent to the absolute spectral form:
\begin{equation}\label{abstract-form-norm-equivalence}
 c_{\mathcal V}\|u\|_{\mathcal V}^2
 \le \sum_{i=1}^n\mu_i|[u,w_i]|^2+[\mathbb L P_+u,P_+u]
 \le C_{\mathcal V}\|u\|_{\mathcal V}^2.
\end{equation}

Consider a strong solution of
\begin{equation}\label{abstract-second-order-equation}
 u_{tt}+\mathbb Lu=\mathcal N(u)
\end{equation}
that conserves an energy of the form
\begin{equation}\label{abstract-energy-expansion}
 \mathcal E(u,u_t)=\frac12\|u_t\|_{\mathcal H}^2+\frac12[\mathbb Lu,u]+\mathcal R_E(u).
\end{equation}
For $u_-:=P_-u=\sum_i a_iw_i$, set
\[
 \mathcal R_G(u):=\sum_{i=1}^n a_i[\mathcal N(u),w_i].
\]
We assume that there is a modulus $\omega(r)\to0$ as $r\to0+$ such that, whenever $\|u\|_{\mathcal V}\le r$,
\begin{equation}\label{abstract-remainder-hypothesis}
 |\mathcal R_E(u)|+|\mathcal R_G(u)|
 \le \omega(r)\|u\|_{\mathcal V}^2.
\end{equation}

\begin{proposition}[Lower-energy escape]\label{abstract-lower-energy-proposition}
Under the preceding assumptions, there exists $r_*>0$ such that, for every
$0<r_0\le r_*$, there are constants $C_*>0$,
$0<\alpha=\alpha(r_0)<1/2$, and $c_0=c_0(r_0)>0$ with the following
property.  Let
\[
 E_0:=\mathcal E(u(0),u_t(0))<0,
 \qquad
 a_0:=2\sqrt{(1-\alpha)\mu_1}.
\]
A strong solution of \eqref{abstract-second-order-equation}
with $\|u(0)\|_{\mathcal V}<r_0$ cannot remain in the ball
$\|u(t)\|_{\mathcal V}<r_0$ for all $t\ge0$.
Writing
\[
 T_{\rm esc}:=\inf\{t\ge0:\|u(t)\|_{\mathcal V}\ge r_0\},
\]
more quantitatively, if the solution remains in that ball on $[0,T]$, then
\begin{equation}\label{abstract-G-absolute-growth}
 G'(t)\ge a_0|G(t)|+c_0|E_0|,
 \qquad
 G(t):=[P_-u(t),u_t(t)].
\end{equation}
Consequently the exit time satisfies
\begin{equation}\label{abstract-escape-time}
 T_{\rm esc}
 \le \frac{2}{a_0}
 \log\!\left(1+\frac{a_0C_*r_0^2}{c_0|E_0|}\right).
\end{equation}
If $G(0)\ge0$, the factor $2$ on the right-hand side of
\eqref{abstract-escape-time} may be omitted, and, as long as the solution
remains in the ball,
\begin{align}
 G(t)&\ge e^{a_0t}G(0)
 +\frac{c_0|E_0|}{a_0}\bigl(e^{a_0t}-1\bigr),
 \label{abstract-G-explicit-lower}\\
 \|P_-u(t)\|_{\mathcal H}^2
 &\ge \|P_-u(0)\|_{\mathcal H}^2
 +\frac{2G(0)}{a_0}\bigl(e^{a_0t}-1\bigr)
 +\frac{2c_0|E_0|}{a_0^2}
 \bigl(e^{a_0t}-1-a_0t\bigr).
 \label{abstract-negative-mode-lower}
\end{align}
Moreover, $\alpha(r_0)\to0$ as $r_0\to0+$.
\end{proposition}

\begin{proof}
Let
\[
 G(u,u_t):=[P_-u,u_t]=\sum_{i=1}^na_i[u_t,w_i].
\]
Differentiating and using \eqref{abstract-second-order-equation} gives
\begin{equation}\label{abstract-G-derivative}
 G'=\sum_{i=1}^n[u_t,w_i]^2+\sum_{i=1}^n\mu_i a_i^2+\mathcal R_G(u).
\end{equation}
For $0<\alpha<1$, split the first two terms into
\[
 G_1:=\sum_i[u_t,w_i]^2+(1-\alpha)\sum_i\mu_i a_i^2,
 \qquad
 G_2:=\alpha\sum_i\mu_i a_i^2+\mathcal R_G(u).
\]
Cauchy's inequality yields
\begin{equation}\label{abstract-G1-bound}
 G_1\ge 2\sqrt{(1-\alpha)\mu_1}\,|G|=a_0|G|.
\end{equation}
Since the energy is negative and the kinetic energy is nonnegative,
\eqref{abstract-energy-expansion} implies
\[
 [\mathbb Lu,u]+2\mathcal R_E(u)<0.
\]
Set
\[
 Q_-:=\sum_i\mu_i a_i^2,
 \qquad Q_+:=[\mathbb Lu_+,u_+]\ge0.
\]
The preceding inequality and \eqref{abstract-remainder-hypothesis} give
\[
 Q_+\le Q_-+2\omega(r_0)\|u\|_{\mathcal V}^2.
\]
Together with \eqref{abstract-form-norm-equivalence}, this implies, after
shrinking $r_0$, that
\[
 \|u\|_{\mathcal V}^2\le C Q_-.
\]
Replace $\omega$ by its nondecreasing envelope and choose
\[
 \alpha=4C\omega(r_0),
\]
with $C$ larger than the constants in the last estimate and in
\eqref{abstract-remainder-hypothesis}.  Shrinking $r_0$ once more ensures
$0<\alpha<1/2$ and
\[
 |\mathcal R_G(u)|\le \frac{\alpha}{2}Q_-.
\]
Consequently,
\begin{equation}\label{abstract-G2-coercive}
 G_2=\alpha Q_-+\mathcal R_G(u)
 \ge \frac{\alpha}{2}Q_-
 \ge c\alpha\|u\|_{\mathcal V}^2.
\end{equation}
The conserved energy also gives
\[
 |E_0|\le C\|u(t)\|_{\mathcal V}^2,
 \qquad
 \|u_t(t)\|_{\mathcal H}^2\le C\|u(t)\|_{\mathcal V}^2,
\]
as long as $\|u(t)\|_{\mathcal V}<r_0$.  Hence, after decreasing $c_0$ if
necessary, \eqref{abstract-G2-coercive} implies
$G_2\ge c_0|E_0|$.  Combining this with
\eqref{abstract-G1-bound} proves \eqref{abstract-G-absolute-growth}.

The same estimates and the continuous embedding
$\mathcal V\hookrightarrow\mathcal H$ yield the bootstrap bound
\begin{equation}\label{abstract-G-upper}
 |G(t)|\le \|P_-u(t)\|_{\mathcal H}\|u_t(t)\|_{\mathcal H}
 \le C_*r_0^2.
\end{equation}
If $G(0)\ge0$, comparison with the solution of
$y'=a_0y+c_0|E_0|$ gives \eqref{abstract-G-explicit-lower}.  Since
\[
 \frac{d}{dt}\|P_-u(t)\|_{\mathcal H}^2=2G(t),
\]
integrating \eqref{abstract-G-explicit-lower} gives
\eqref{abstract-negative-mode-lower}.  Comparing
\eqref{abstract-G-explicit-lower} with \eqref{abstract-G-upper} shows that
the solution must exit no later than
\[
 \frac1{a_0}\log\!\left(
 1+\frac{a_0C_*r_0^2}{c_0|E_0|}\right).
\]

If $G(0)<0$, then while $G<0$, \eqref{abstract-G-absolute-growth} reads
$G'+a_0G\ge c_0|E_0|$.  Together with
$G(0)\ge-C_*r_0^2$, this shows that either the solution exits or $G$ reaches
zero within the same logarithmic time.  Starting from the first time at
which $G=0$, the preceding $G\ge0$ argument requires at most one additional
interval of the same length.  This proves \eqref{abstract-escape-time} and
the escape conclusion.
\end{proof}

\begin{remark}
The finite-dimensional phase portrait is a hyperbolic saddle, not a saddle-node bifurcation.  Initial data with $G(0)<0$ may first move toward the equilibrium along a stable direction; the lower-energy condition prevents the orbit from remaining on the center-stable side and eventually forces $G$ to increase.  Estimate \eqref{abstract-escape-time} bounds this waiting stage and the subsequent growing stage in terms of the conserved negative energy.  Theorem~\ref{Main thm} implements the same mechanism with the precise weighted norms of the Euler--Poisson problem.
\end{remark}

\begin{remark}[Phase-plane interpretation]
On a single negative eigendirection, the linearized equation has the form
\[
 a''-\mu a=0,
 \qquad b:=a',
\]
with growing and decaying lines $b=\sqrt\mu\,a$ and
$b=-\sqrt\mu\,a$, respectively.  The reduced Lyapunov functional is
$G=ab$.  Thus $G(0)\ge0$ corresponds to data whose hyperbolic motion is
already oriented toward a growing branch, whereas $G(0)<0$ allows an
initially decaying, or stable-dominated, stage.  The lower-energy condition
forces $G$ to increase; hence the orbit either leaves the small neighborhood
before reaching $G=0$, or crosses to $G\ge0$ and subsequently grows
exponentially.  The following diagram is only schematic; the center
variables and nonlinear coupling are suppressed.

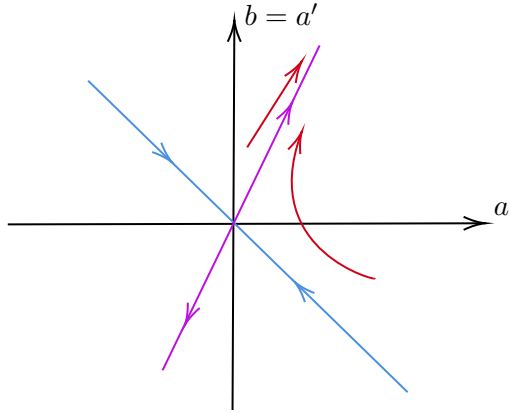
\begin{center}
\tikzset{every picture/.style={line width=0.75pt}} %set default line width to 0.75pt        

\begin{tikzpicture}[x=0.75pt,y=0.75pt,yscale=-1,xscale=1]
%uncomment if require: \path (0,235); %set diagram left start at 0, and has height of 235

%Straight Lines [id:da7965172321696647] 
\draw    (307.86,137.63) -- (545.29,137.08) ;
\draw [shift={(547.29,137.07)}, rotate = 179.87] [color={rgb, 255:red, 0; green, 0; blue, 0 }  ][line width=0.75]    (10.93,-3.29) .. controls (6.95,-1.4) and (3.31,-0.3) .. (0,0) .. controls (3.31,0.3) and (6.95,1.4) .. (10.93,3.29)   ;
%Straight Lines [id:da3296282784668546] 
\draw    (420.52,230.78) -- (421.43,37.03) ;
\draw [shift={(421.44,35.03)}, rotate = 90.27] [color={rgb, 255:red, 0; green, 0; blue, 0 }  ][line width=0.75]    (10.93,-3.29) .. controls (6.95,-1.4) and (3.31,-0.3) .. (0,0) .. controls (3.31,0.3) and (6.95,1.4) .. (10.93,3.29)   ;
%Straight Lines [id:da5827747508823141] 
\draw [color={rgb, 255:red, 74; green, 144; blue, 226 }  ,draw opacity=1 ]   (348.13,65.52) -- (508.3,221.86) ;
%Straight Lines [id:da3687131043775915] 
\draw [color={rgb, 255:red, 189; green, 16; blue, 224 }  ,draw opacity=1 ]   (385.46,210.87) -- (464.17,47.83) ;
%Straight Lines [id:da44485451384253394] 
\draw [color={rgb, 255:red, 74; green, 144; blue, 226 }  ,draw opacity=1 ]   (388.49,104.81) -- (388.89,105.2) ;
\draw [shift={(390.32,106.6)}, rotate = 224.31] [color={rgb, 255:red, 74; green, 144; blue, 226 }  ,draw opacity=1 ][line width=0.75]    (10.93,-3.29) .. controls (6.95,-1.4) and (3.31,-0.3) .. (0,0) .. controls (3.31,0.3) and (6.95,1.4) .. (10.93,3.29)   ;
%Straight Lines [id:da41215331169668223] 
\draw [color={rgb, 255:red, 74; green, 144; blue, 226 }  ,draw opacity=1 ]   (453.86,168.78) -- (453.02,167.9) ;
\draw [shift={(451.64,166.45)}, rotate = 46.34] [color={rgb, 255:red, 74; green, 144; blue, 226 }  ,draw opacity=1 ][line width=0.75]    (10.93,-3.29) .. controls (6.95,-1.4) and (3.31,-0.3) .. (0,0) .. controls (3.31,0.3) and (6.95,1.4) .. (10.93,3.29)   ;
%Straight Lines [id:da9076056216126517] 
\draw [color={rgb, 255:red, 189; green, 16; blue, 224 }  ,draw opacity=1 ]   (449.57,78.14) ;
\draw [shift={(450.43,76.43)}, rotate = 116.57] [color={rgb, 255:red, 189; green, 16; blue, 224 }  ,draw opacity=1 ][line width=0.75]    (10.93,-3.29) .. controls (6.95,-1.4) and (3.31,-0.3) .. (0,0) .. controls (3.31,0.3) and (6.95,1.4) .. (10.93,3.29)   ;
%Straight Lines [id:da9878540010859495] 
\draw [color={rgb, 255:red, 189; green, 16; blue, 224 }  ,draw opacity=1 ]   (397.29,186.71) ;
\draw [shift={(396.61,188.22)}, rotate = 294.15] [color={rgb, 255:red, 189; green, 16; blue, 224 }  ,draw opacity=1 ][line width=0.75]    (10.93,-3.29) .. controls (6.95,-1.4) and (3.31,-0.3) .. (0,0) .. controls (3.31,0.3) and (6.95,1.4) .. (10.93,3.29)   ;
%Curve Lines [id:da9460481782610759] 
\draw [color={rgb, 255:red, 208; green, 2; blue, 27 }  ,draw opacity=1 ]   (491.46,164.67) .. controls (494.17,166.44) and (433.65,151.57) .. (454.64,92.76) ;
\draw [shift={(455.3,90.96)}, rotate = 110.78] [color={rgb, 255:red, 208; green, 2; blue, 27 }  ,draw opacity=1 ][line width=0.75]    (10.93,-3.29) .. controls (6.95,-1.4) and (3.31,-0.3) .. (0,0) .. controls (3.31,0.3) and (6.95,1.4) .. (10.93,3.29)   ;
%Straight Lines [id:da7779572208223102] 
\draw [color={rgb, 255:red, 208; green, 2; blue, 27 }  ,draw opacity=1 ]   (427.85,99) -- (454.23,57.37) ;
\draw [shift={(455.3,55.68)}, rotate = 122.36] [color={rgb, 255:red, 208; green, 2; blue, 27 }  ,draw opacity=1 ][line width=0.75]    (10.93,-3.29) .. controls (6.95,-1.4) and (3.31,-0.3) .. (0,0) .. controls (3.31,0.3) and (6.95,1.4) .. (10.93,3.29)   ;

\draw (550, 125) node [anchor=north west][inner sep=0.75pt]    {$a$};
\draw (425, 25) node [anchor=north west][inner sep=0.75pt]    {$b=a'$};

\end{tikzpicture}
\end{center}
\end{remark}

\begin{remark}
For a separable first-order Hamiltonian system
\[
 \binom{u_t}{v_t}=\begin{pmatrix}0&S\\-S'&0\end{pmatrix}
 \binom{H_2'(u)}{Tv},
\]
differentiating $u_t=STv$ and using $v_t=-S'H_2'(u)$ formally gives
\[
 u_{tt}+STS'H_2'(u)=0.
\]
If $STS'$ is positive and invertible, applying $(STS')^{-1}$ yields a
second-order equation of the form considered above, in the Hilbert structure
induced by $(STS')^{-1}$.  In applications with unbounded operators, all
domains and duality maps must be specified separately.  The Euler--Poisson
realization below is carried out directly in the weighted spaces $Y_\mu$ and
$Z_\mu$.
\end{remark}

\section{Main results}\label{main-results-section}
We first state two conditional instability theorems for the general pressure
laws satisfying \eqref{property of pressure 1}--\eqref{property of pressure 2}.
The conditional solution class used below is specified by the following
interface.

\begin{unnumbereddefinition}[Strong radial solution]
Let $I=[0,T]$.  A pair $(X,U)$ is called a \emph{strong radial solution} on
$I$ if the following properties hold.
\begin{enumerate}[(S1)]
\item With $\widetilde X:=X-\mathrm{id}$,
\[
 \widetilde X\in C(I;Z_\mu),\qquad U\in C(I;Y_\mu^*),
\]
and, for each $t\in I$, $X(t)$ is an orientation-preserving radial
bi-Lipschitz map fixing the origin.  The map
\[
 t\longmapsto \mathfrak d(X(t))
 :=\|X(t)-\mathrm{id}\|_{W^{1,\infty}(S_\mu)}
\]
is continuous.
\item For every radial $w\in Z_\mu$, the scalar functions
\[
 t\longmapsto [\widetilde X(t),w]_{Y_\mu},
 \qquad
 t\longmapsto \langle U(t),w\rangle
\]
are locally absolutely continuous and satisfy, for almost every $t\in I$,
\begin{equation}\label{weak-Hamilton-equations-interface}
 \frac d{dt}[\widetilde X(t),w]_{Y_\mu}
 =\langle U(t),w\rangle,
 \qquad
 \frac d{dt}\langle U(t),w\rangle
 =-\left\langle\frac{\delta H}{\delta X}(X(t)),w\right\rangle.
\end{equation}
Here the first bracket is the $Y_\mu$ inner product and the other brackets
are the natural primal--dual pairings.  Equivalently,
$X_t=\rho_\mu^{-1}U$ and the momentum equation hold in the duality sense
used below.
\item The Hamiltonian is finite and conserved:
\[
 H(X(t),U(t))=H(X(0),U(0)),\qquad t\in I.
\]
\item Whenever $\mathfrak d(X(t))\le\delta_*$, with $\delta_*$ as in
Lemma~\ref{Estimation of dH/dX-Lmu}, the radial deformation factors
\[
 a=\partial_r\chi,\qquad b=\chi/r,\qquad J=ab^2,
 \qquad X(x)=\chi(|x|)\frac{x}{|x|},
\]
lie in a fixed compact subset of $(0,\infty)$, and the variational identities
in Sections~2 and~7 hold.  For sufficiently small $\delta_*$, the compactness
assertion follows automatically from (S1).
\end{enumerate}
\end{unnumbereddefinition}

Theorems~\ref{Main thm} and~\ref{Main result when energy is higher} use
only the properties in the preceding definition; hence any present or future
local theory producing this interface may be used in the conditional results.
In particular, the deformation norm
\begin{equation}\label{mixed-deformation-norm}
 \mathfrak d(X):=\|X-\mathrm{id}\|_{W^{1,\infty}(S_\mu)}
\end{equation}
is finite and continuous on the interval under consideration.

\begin{theorem}[Conditional lower-energy instability]\label{Main thm}
Assume
\[
 n^u(\mu)=n^-(L_\mu)-i_\mu>0,
 \qquad M'(\mu)\ne0.
\]
There are $\delta_0>0$, $D>0$, and a function
$h:(0,\delta_0]\to(0,\infty)$ satisfying
\begin{equation}\label{h-delta-limit}
 h(\delta)=1+o(1)\qquad(\delta\to0+)
\end{equation}
such that the following holds for every $0<\delta\le\delta_0$.  Let
$(X(t),U(t))$ be a strong radial solution with
\[
 H_0:=H(X(0),U(0))<H(\mathrm{id},0)=0
\]
and suppose that it is defined on $[0,T_\delta]$, where
\begin{equation}\label{lower-energy-escape-time}
 T_\delta=\frac{h(\delta)}{\sqrt{\mu_1}}
 \log\!\left(\frac{D\delta}{|H_0|}\right).
\end{equation}
Here $-\mu_1<0$ is the negative eigenvalue of
$\widetilde{\mathbb L}_\mu$ closest to the origin.  If the logarithm in
\eqref{lower-energy-escape-time} is positive, then
\begin{equation}\label{conditional-lower-energy-escape}
 \sup_{0\le t\le T_\delta}\mathfrak d(X(t))\ge\delta.
\end{equation}

Let $w_1,\dots,w_n$ be an orthonormal basis of the negative eigenspace of
$\widetilde{\mathbb L}_\mu$, with
$\widetilde{\mathbb L}_\mu w_i=-\mu_iw_i$, and define
\begin{equation}\label{definition-of-V-main-theorem}
 V(t):=\sum_{i=1}^n[X(t)-\mathrm{id},w_i]\,
       \langle U(t),w_i\rangle_{L^2}.
\end{equation}
If, in addition, $V(0)\ge0$, then there is a constant $C_\delta>0$ such that
\begin{equation}\label{growth of Y_mu norm}
 \|X(t)-\mathrm{id}\|_{Y_\mu}^2
 \ge C_\delta|H_0|
 \left[
 e^{\frac{2\sqrt{\mu_1}}{h(\delta)}t}
 -1-\frac{2\sqrt{\mu_1}}{h(\delta)}t
 \right],
 \qquad 0\le t\le T_\delta^*,
\end{equation}
where
\begin{equation}\label{definition-T-star}
 T_\delta^*:=\sup\left\{t\ge0:\mathfrak d(X(t))<\delta\right\}.
\end{equation}

\end{theorem}

\begin{unnumberedremark}
The condition that the logarithm in \eqref{lower-energy-escape-time} be
positive is only a bookkeeping condition ensuring that the displayed time is
nonnegative.  If $\mathfrak d(X(0))\ge\delta$, then
\eqref{conditional-lower-energy-escape} already holds at $t=0$.  Otherwise,
the estimates used in the proof, in particular \eqref{H-controlled-by-X}
and the small-deformation estimate preceding \eqref{V-bootstrap-upper}, give
$|H_0|\le C\delta^2$.  After decreasing $\delta_0$ so that
$C\delta_0<D$, one has $D\delta/|H_0|>1$ automatically.  Thus no small
lower-energy perturbation is excluded by this condition.
\end{unnumberedremark}

\begin{unnumberedremark}[Eulerian interpretation]
While $\mathfrak d(X)<\delta$ with $\delta$ small, $X$ and $X^{-1}$ are
uniformly bi-Lipschitz, $\det DX$ stays close to one, and
\[
 \rho(t,X(t,x))=\frac{\rho_\mu(x)}{\det DX(t,x)}.
\]
Consequently, the Eulerian density and the moving support remain in a
controlled neighborhood of the equilibrium density and ball throughout the
bootstrap interval.  The conclusions of Theorems~\ref{Main thm} and
\ref{Main result when energy is higher} assert that this controlled
Lagrangian--Eulerian neighborhood cannot persist to the stated logarithmic
time.  We do not claim here a separate quantitative lower bound in a
specific Eulerian norm.
\end{unnumberedremark}

\begin{remark}
The effective amplitude growth rate in \eqref{growth of Y_mu norm} is
$\sqrt{\mu_1}/h(\delta)$, which tends to the least positive linear growth
rate $\sqrt{\mu_1}$ as $\delta\to0+$.  The exponent in
\eqref{growth of Y_mu norm} is twice this amplitude rate because the
left-hand side is a squared norm.  The theorem gives escape for every
lower-energy solution satisfying the stated existence hypothesis, but the
explicit lower bound \eqref{growth of Y_mu norm} is stated only for
$V(0)\ge0$.
\end{remark}

\begin{remark}\label{positive-V-examples}
The index assumption implies that the second variation is negative on a
nonzero dynamically accessible radial direction.  Choosing a smooth
displacement $X_0-\mathrm{id}$ sufficiently close to such a direction and
taking $U_0=0$, Corollary~\ref{Estimation of F_1} gives
\[
 H(X_0,U_0)<H(\mathrm{id},0)
\]
for all sufficiently small nonzero amplitudes.  In this elementary class
$V(0)=0$, so the explicit bound in Theorem~\ref{Main thm} applies.

One may also obtain $V(0)>0$ explicitly from the least unstable mode.  Let
$\widetilde{\mathbb L}_\mu w_1=-\mu_1w_1$ and normalize
$\|w_1\|_{Y_\mu}=1$.  For fixed $\theta>0$ and sufficiently small $M>0$,
set
\begin{equation}\label{explicit-positive-V-data}
 X_0-\mathrm{id}=M(1+\theta)w_1,
 \qquad
 U_0=M\sqrt{\mu_1}\,\rho_\mu w_1.
\end{equation}
Then
\[
 V(0)=M^2(1+\theta)\sqrt{\mu_1}>0.
\]
Moreover, Corollary~\ref{Estimation of F_1} yields
\begin{align*}
 H(X_0,U_0)
 &=\frac12M^2\mu_1
   -\frac12M^2\mu_1(1+\theta)^2+o(M^2)\\
 &=-M^2\mu_1\left(\theta+\frac{\theta^2}{2}\right)+o(M^2)<0
\end{align*}
when $M$ is small.  Thus \eqref{explicit-positive-V-data} gives lower-energy
data for which the exponential lower bound begins at $t=0$.  If the local
class requires additional smoothness, one may replace $w_1$ by a sufficiently
close smooth radial approximation; both strict inequalities persist.
\end{remark}

We next formulate the instability mechanism based on an unstable spectral component.  Let
\begin{equation}\label{linearized-generator-definition}
 \mathcal A=
 \begin{pmatrix}
 0&\rho_\mu^{-1}\\
 -\widetilde L_\mu&0
 \end{pmatrix}
\end{equation}
be the linearized Hamiltonian generator on $\vec Z_\mu=Z_\mu\times Y_\mu^*$.  Since $M'(\mu)\ne0$, zero is not in the spectrum.  The positive real eigenvalues are the finitely many numbers $\sqrt{\mu_i}$.  We define the unstable Riesz projection by
\begin{equation}\label{Riesz-projection-definition}
 P_u:=\frac1{2\pi i}\int_\Gamma(z-\mathcal A)^{-1}\,dz,
\end{equation}
where $\Gamma$ is a positively oriented contour enclosing precisely the positive real spectrum of $\mathcal A$.  The projection is independent of the choice of such a contour.

\begin{theorem}[Conditional instability from an unstable component]\label{Main result when energy is higher}
Assume $n^u(\mu)>0$ and $M'(\mu)\ne0$.  For every fixed cone aperture
$\kappa\in(0,1]$, there exist $\bar\delta_\kappa>0$ and a function
\[
 h_\kappa:(0,\bar\delta_\kappa]\longrightarrow(0,\infty),
 \qquad h_\kappa(\delta)=1+o(1)\quad(\delta\to0+),
\]
with the following property.  For every $0<\delta\le\bar\delta_\kappa$
there is a constant $D_2=D_2(\kappa,\delta)>0$ such that, if
\[
 Z(t):=(X(t)-\mathrm{id},U(t))
\]
is a strong radial solution satisfying
\begin{equation}\label{unstable-cone-initial-condition}
 0<\|Z(0)\|_{\vec Z_\mu}\le\delta,
 \qquad
 \|P_uZ(0)\|_{\vec Z_\mu}\ge\kappa\|Z(0)\|_{\vec Z_\mu},
\end{equation}
and if the solution is defined on $[0,T]$, where
\begin{equation}\label{projection-escape-time}
 T=\frac{h_\kappa(\delta)}{\sqrt{\mu_1}}
 \log\!\left(\frac{D_2(\kappa,\delta)}{\|Z(0)\|_{\vec Z_\mu}}\right),
\end{equation}
then
\begin{equation}\label{conditional-projection-escape}
 \sup_{0\le t\le T}\mathfrak d(X(t))\ge\delta.
\end{equation}
The constants are uniform over initial data satisfying
\eqref{unstable-cone-initial-condition} for the fixed values of $\kappa$ and
$\delta$.  In particular, the logarithmic-time prefactor
\[
 D_1(\kappa,\delta):=\frac{h_\kappa(\delta)}{\sqrt{\mu_1}}
\]
satisfies
\[
 \lim_{\delta\to0+}D_1(\kappa,\delta)=\frac1{\sqrt{\mu_1}}.
\]
\end{theorem}

\begin{remark}[Relation between the two instability regions]
The parameter $\kappa$ measures the relative size of the unstable component
of the initial perturbation.  Because $P_u$ is a Riesz projection and need
not be orthogonal in the original norm, $\kappa$ is best interpreted as the
aperture of an unstable cone rather than literally as the cosine of an
angle.  The admissible threshold $\bar\delta_\kappa$ may decrease as
$\kappa\to0$.

Theorem~\ref{Main result when energy is higher} and
Theorem~\ref{Main thm} describe genuinely different sets of data.  In one
hyperbolic mode, with displacement and velocity coordinates $(a,b)$, the
quadratic energy is proportional to $b^2-\mu a^2$; hence the negative-energy
region lies between the two zero-energy lines $b=\pm\sqrt\mu\,a$.  The
unstable cone is instead a neighborhood of the growing line
$b=\sqrt\mu\,a$.  It contains data of positive as well as negative energy,
whereas the negative-energy region contains data whose growing projection
may be arbitrarily small.  The following schematic illustrates the two
regions; center variables and nonlinear corrections are suppressed.

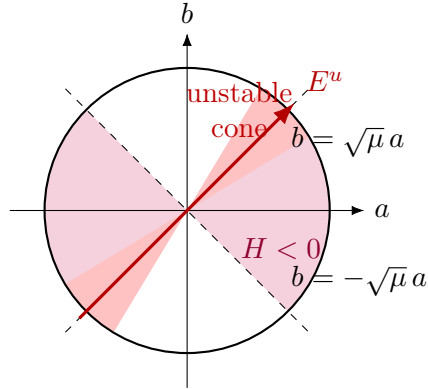
\begin{center}
\begin{tikzpicture}[scale=0.92,>=Latex]
  \begin{scope}
    \clip (0,0) circle (2.05);
    % negative quadratic-energy sectors |b|<|a|
    \fill[purple!18] (0,0) -- (3,3) -- (3,-3) -- cycle;
    \fill[purple!18] (0,0) -- (-3,3) -- (-3,-3) -- cycle;
    % cone around the growing line b=a, including both signs
    \fill[red!24] (0,0) -- (2.9,1.75) -- (1.75,2.9) -- cycle;
    \fill[red!24] (0,0) -- (-2.9,-1.75) -- (-1.75,-2.9) -- cycle;
  \end{scope}
  \draw[thick] (0,0) circle (2.05);
  \draw[->] (-2.55,0) -- (2.55,0) node[right] {$a$};
  \draw[->] (0,-2.55) -- (0,2.55) node[above] {$b$};
  \draw[dashed] (-1.75,-1.75) -- (1.75,1.75);
  \draw[dashed] (-1.75,1.75) -- (1.75,-1.75);
  \draw[red!70!black,very thick,->] (-1.55,-1.55) -- (1.55,1.55)
    node[above right] {$E^u$};
  \node[purple!75!black] at (1.35,-0.55) {$ H<0$};
  \node[red!75!black,align=center] at (0.75,1.45) {unstable\\cone};
  \node[anchor=north west] at (1.35,1.35) {$b=\sqrt\mu\,a$};
  \node[anchor=south west] at (1.35,-1.38) {$b=-\sqrt\mu\,a$};
\end{tikzpicture}
\end{center}
Thus the two theorems complement one another, and neither contains the
other.  A pure growing eigenvector satisfies $P_uZ_0=Z_0$ and therefore
belongs to the cone with $\kappa=1$ (and hence with every
$0<\kappa\le1$); in this sense the leading-mode family used by Jang is a
special subset of the cone data.
\end{remark}

\subsection{The local well-posedness input}
The conditional theorems require a strong radial solution on a logarithmic
time interval.  In the polytropic case this is supplied by the
physical-vacuum local theory.  Gu and Lei treat the three-dimensional
Euler--Poisson problem in a high-order weighted Lagrangian class
\cite{GU2016662}.  Since our instability arguments are radial, we use the
radial theory of Luo, Xin, and Zeng \cite{Xinzhouping}, which is formulated
on a ball, includes self-gravitation, and gives uniqueness in the range
needed below.  We record only the consequences used in the instability
argument; the full weighted energy is recalled in
Appendix~\ref{appendix-LXZ-energy}.

\begin{theorem}\label{Well-posedness}
Let $P(\rho)=C_\gamma\rho^\gamma$ with $6/5<\gamma<4/3$.  In the radial
Lagrangian variables of \cite{Xinzhouping}, rescale the reference ball to
$I=(0,1)$, write $r=r(x,t)$ for the physical radius and $v=r_t$, and let
$\widetilde{\mathcal E}_\gamma(v,t)$ denote the high-order weighted energy
in \cite[(9.1)]{Xinzhouping}; see
Appendix~\ref{appendix-LXZ-energy}.  For smooth radial initial data satisfying
the physical-vacuum condition of \cite[(3.3)--(3.4)]{Xinzhouping} and the
finite-energy compatibility conditions of the radial construction, the
following properties hold near a Lane--Emden equilibrium.
\begin{enumerate}[(i)]
\item The data generate a unique classical Lagrangian solution on a
nontrivial time interval.

\item In the identity initial chart, define the centered high-order size
\begin{equation}\label{definition-N-gamma}
 \mathcal N_\gamma(t)^2
 :=\|r(\cdot,t)-x\|_{H^2(I)}^2
   +\widetilde{\mathcal E}_\gamma(v,t).
\end{equation}
Then $t\mapsto\mathcal N_\gamma(t)$ is continuous on the local existence
interval, $\mathcal N_\gamma=0$ at the equilibrium, and, in a fixed sufficiently
small neighborhood,
\begin{equation}\label{high-energy-controls-bootstrap}
 \mathfrak d(X(t))
 =\|X(t)-\mathrm{id}\|_{W^{1,\infty}(S_\mu)}
 \le C\mathcal N_\gamma(t).
\end{equation}

\item As long as the high-order energy and Jacobian bounds remain controlled,
and the physical-vacuum constant stays nondegenerate, the local construction
can be restarted with a lifespan bounded below in terms of these quantities.  Thus a solution that remains in a fixed controlled
$\mathcal N_\gamma$-neighborhood can be continued by finitely many restarts
across any prescribed finite time interval.
\end{enumerate}
If $X(0)\ne\mathrm{id}$, the same statements are used after taking
$X(0)(S_\mu)$ as the physical initial domain and relabeling the radial flow.
A smooth radial diffeomorphism sufficiently close to the identity preserves
the physical-vacuum behavior under this push-forward; see
Appendix~\ref{appendix-LXZ-energy}.
\end{theorem}

\begin{remark}
The particular distance \eqref{definition-N-gamma} is not canonical.  The
proof below uses only three of its properties: small compatible initial data
have small $\mathcal N_\gamma$; the estimate
\eqref{high-energy-controls-bootstrap} controls the deformation norm; and
bounded $\mathcal N_\gamma$ together with the geometric vacuum bounds gives
a uniform restart time.  In particular,
\begin{equation}\label{bootstrap-forces-N-gamma}
 \mathcal N_\gamma(t)\ge C^{-1}\mathfrak d(X(t)),
\end{equation}
so escape in either conditional theorem implies escape in the high-order
local-solution topology.
\end{remark}

\begin{remark}\label{general-EOS-wp-remark}
Theorem~\ref{Well-posedness} is used only for polytropic pressure laws.  For
a general pressure satisfying \eqref{property of pressure 1}--
\eqref{property of pressure 2}, one expects an analogous local theory to
require additional smoothness and weighted control of higher pressure
derivatives near vacuum.  Establishing such a theory, including the
compatibility and uniqueness statements needed for restart, is a separate
problem and is not assumed here.
\end{remark}

\begin{theorem}[Unconditional nonlinear instability in the polytropic class]
\label{Main thm 0}
Let $P(\rho)=C_\gamma\rho^\gamma$ with $6/5<\gamma<4/3$, and let
$(\rho_\mu,0)$ be a Lane--Emden equilibrium.  Then $n^u(\mu)=1$ and
$M'(\mu)\ne0$.  There is a fixed threshold $\nu_{\mathrm{low}}>0$ for the
lower-energy statement, and for every fixed $\kappa\in(0,1]$ there is a
threshold $\nu_{\mathrm{cone}}(\kappa)>0$ for the unstable-cone statement,
both measured in the high-order distance \eqref{definition-N-gamma}, with
the following properties.
\begin{enumerate}[(i)]
\item For every $\varepsilon_*>0$, there are smooth compatible radial data
such that
\[
 0<\mathcal N_\gamma(0)<\varepsilon_*,
 \qquad H(X(0),U(0))<H(\mathrm{id},0),
\]
and the corresponding classical solution satisfies
\[
 \mathcal N_\gamma(t_*)\ge\nu_{\mathrm{low}}
\]
for some $t_*$ no later than the time $T_\delta$ in
\eqref{lower-energy-escape-time}, with a fixed sufficiently small escape
threshold $\delta$.

\item For every fixed cone aperture $\kappa\in(0,1]$ and every
$\varepsilon_*>0$, there are smooth compatible radial data satisfying
\eqref{unstable-cone-initial-condition}, with
$0<\mathcal N_\gamma(0)<\varepsilon_*$, for which
\[
 \mathcal N_\gamma(t_*)\ge\nu_{\mathrm{cone}}(\kappa)
\]
no later than the time in \eqref{projection-escape-time}, with a fixed
sufficiently small $\delta\le\bar\delta_\kappa$.
\end{enumerate}
Thus the equilibrium is nonlinearly unstable in the same high-order
physical-vacuum topology in which the classical solution is constructed.
\end{theorem}

\begin{remark}
The admissible data in Theorem~\ref{Main thm 0} can be obtained by smooth,
compatible approximation.  For part~(i), approximate a negative direction
of the displacement Hessian and take $U(0)=0$; the strict lower-energy
inequality persists.  For part~(ii), approximate an unstable eigenvector;
after slightly decreasing the cone aperture, the strict cone inequality
persists.  Scaling the amplitude then makes $\mathcal N_\gamma(0)$
arbitrarily small.
\end{remark}

\begin{remark}
Jang's nonlinear analysis \cite{Jang_2013_nonlinear} provides another
natural high-order size in the Lane--Emden setting.  If
$\widetilde E_J$ denotes her full weighted energy, then
$\mathcal N_{\gamma,J}:=\widetilde E_J^{1/2}$ is centered at the equilibrium
and her weighted Hardy estimates imply
$\|X-\mathrm{id}\|_{W^{1,\infty}}\lesssim\mathcal N_{\gamma,J}$.
If the corresponding local theory is supplied with uniqueness and a uniform
restart criterion, the stopping-time argument below can therefore be run
with $\mathcal N_{\gamma,J}$ in place of $\mathcal N_\gamma$ and gives
strong-to-strong escape in that topology.

Jang's leading-mode argument proves more: the initial perturbation is small
in the high-order energy, while escape is detected in the weaker
zeroth-order instantaneous energy $E_J^0$ at the sharp time
\[
 T^\delta=\frac1{\sqrt{\mu_0}}\log\frac{2\theta_0}{\delta},
\]
so her result is strong-to-weak.  

Jang's strong-to-weak instability result already applies to a special
subset of the cone data.  Indeed, for a pure growing eigenmode one has
\(P_u Z_0=Z_0\), so the cone condition is satisfied for every
\(0<\kappa\le1\).  More generally, one expects Jang's Duhamel argument
to extend to finite-dimensional, spectrally localized cone data, and
also to the explicit lower-energy family in
Remark~\ref{positive-V-examples}, because in these cases the initial
high-order norm can be quantitatively controlled by the amplitudes of
finitely many unstable modes.

For the full cone class in Theorem~\ref{Main thm 0}, however, the
present argument does not by itself yield strong-to-weak instability.
The cone condition controls the size of the unstable projection only in
the lower-order phase norm, and does not prevent the center-stable
component from being large relative to it in the high-order
physical-vacuum topology.  Consequently, the high-order norm
\(\mathcal N_\gamma\) used for continuation may reach its bootstrap
threshold before the lower-order norm reaches a fixed escape size.
A strong-to-weak extension for the full cone would therefore require
additional high-order estimates valid up to the lower-order escape time,
or, alternatively, a stronger cone assumption that quantitatively
controls the high-order center-stable component by the unstable
amplitude.  Similar difficulties arise for the full lower-energy class.
These refinements are not needed for Theorem~\ref{Main thm 0}.
\end{remark}

\section{Technical Preliminaries}
In this section, we derive the weighted estimates used in the instability arguments.
\begin{lemma}[Endpoint Hardy inequalities, \cite{Hardy}]\label{Hardy's ineq}
Let $k<1$ and $a>0$.
\begin{enumerate}[(i)]
\item If
\begin{equation}\label{Hardy's ineq 1}
 \int_0^a r^k\bigl(|f(r)|^2+|f'(r)|^2\bigr)\,dr<\infty,
\end{equation}
then the trace $f(0+)$ exists and
\begin{equation}\label{Hardy's ineq 2}
 \int_0^a r^{k-2}|f(r)-f(0+)|^2\,dr
 \le C\int_0^a r^k|f'(r)|^2\,dr.
\end{equation}
\item More generally, if $0\le b<a$ and
\[
 \int_b^a(a-r)^k\bigl(|f(r)|^2+|f'(r)|^2\bigr)\,dr<\infty,
\]
then the trace $f(a-)$ exists and the reflected inequality
\begin{equation}\label{Hardy-right-endpoint}
 \int_b^a(a-r)^{k-2}|f(r)-f(a-)|^2\,dr
 \le C\int_b^a(a-r)^k|f'(r)|^2\,dr
\end{equation}
holds.  In particular, when the corresponding endpoint trace vanishes, the
difference terms in \eqref{Hardy's ineq 2} and
\eqref{Hardy-right-endpoint} may be replaced by $f$.
\end{enumerate}
The right-endpoint statement follows from the left-endpoint one by the
change of variables $s=a-r$.  Both forms are valid for every $k<1$,
including negative $k$.
\end{lemma}
\begin{lemma} \label{X Y_mu less than BAX X_mu}
There exists $C>0$ such that, for every $X\in Z_\mu$,
\begin{equation} \label{X Y_mu less than BAX X_mu updated}
    \int_{S_\mu} \rho_\mu^{2-\gamma_0} |X|^2dx \leq C\|B_\mu A_\mu X\|_{X_\mu}^2, 
\end{equation}
and in particular, as $\gamma_0\in (\frac 65, 2)$,
    \begin{equation}\label{X-Y-from-BAX}
        \|X\|_{Y_\mu}\leq C\|B_\mu A_\mu X\|_{X_\mu}.
    \end{equation}
\end{lemma}
\begin{remark}
    This lemma implies that, for $X\in Z_\mu$, $\|B_\mu A_\mu X\|_{X_\mu}$ and $\|X\|_{Z_\mu}$ are equivalent norms.
\end{remark}
\begin{proof}
We first write $X(x)=v(r)\frac{x}{|x|}$. A direct calculation in radial coordinates gives
\[
 \operatorname{div}(\rho_\mu X)
 =\frac1{r^2}\partial_r(r^2\rho_\mu v).
\]
We apply Hardy's inequality to $f=r^2 \rho_\mu v$, with $k=-2$ on $[0,\frac{R_\mu}{2}]$ and $k=\frac{\gamma_0-2}{\gamma_0-1}$ on $[\frac{R_\mu}{2},R_\mu]$.  We first verify the integrability hypotheses of Lemma~\ref{Hardy's ineq} on both intervals.  Since $X\in Y_\mu$, 
\[
\int_0^{\frac{R_\mu}{2}} r^{-2}|f|^2dr\leq \int_0^{R_\mu} \rho_\mu^2 |v|^2 r^2dr\lesssim \int_{S_\mu}\rho_\mu |X|^2dx=\|X\|_{Y_\mu}^2<+\infty.
\]
Next,
\[
\begin{aligned}
  \int_0^{\frac{R_\mu}2}r^{-2} |f^\prime|^2 dr= &\int_0^{\frac{R_\mu}2}r^{-2} |\partial_r (r^2 \rho_\mu v)|^2 dr\leq \int_0^{R_\mu} r^2 |\nabla\cdot (\rho_\mu X)|^2dr \\
   \lesssim  &\int_{S_\mu}\Phi^{\prime\prime}(\rho_\mu)|\nabla\cdot (\rho_\mu X)|^2dx=\|B_\mu A_\mu X\|_{X_\mu}^2<+\infty, 
\end{aligned}
\]
where we used the fact that $\frac{1}{\Phi^{\prime\prime}(\rho_\mu)}$ is bounded. On $[\frac{R_\mu}2, R_\mu]$, by the two-sided estimate $\rho_\mu\asymp(R_\mu-r)^{\frac{1}{\gamma_0-1}}$, it follows that 
\[
\begin{aligned}
&\int_{\frac{R_\mu}2}^{R_\mu} (R_\mu -r)^{\frac{\gamma_0-2}{\gamma_0-1}} |f|^2 dr=
\int_{\frac{R_\mu}2}^{R_\mu} (R_\mu -r)^{\frac{\gamma_0-2}{\gamma_0-1}} \rho_\mu^2 |v|^2 r^2 dr \\ 
\lesssim  &\int_{0}^{R_\mu} (R_\mu -r) \rho_\mu |v|^2 r^2 dr \lesssim \int_{S_\mu} \rho_\mu |X|^2 dx <+\infty.
\end{aligned}
\]
Using $\Phi^{\prime\prime}(\rho_\mu)\asymp(R_\mu-r)^{\frac{\gamma_0-2}{\gamma_0-1}}$, we have
\[
\begin{aligned}
&\int_{\frac{R_\mu}2}^{R_\mu} (R_\mu -r)^{\frac{\gamma_0-2}{\gamma_0-1}} |f^\prime|^2dr \lesssim \int_{0}^{R_\mu} \Phi^{\prime \prime}(\rho_\mu) |\partial_r (r^2 \rho_\mu v)|^2 dr\\
\lesssim&  \int_{0}^{R_\mu} \Phi^{\prime \prime}(\rho_\mu) |\nabla\cdot(\rho_\mu X)|^2 r^2 dr=\int_{S_\mu} \Phi^{\prime \prime}(\rho_\mu) |\nabla\cdot(\rho_\mu X)|^2dx<+\infty.
\end{aligned} 
\]
Thus the hypotheses of Hardy's inequality are satisfied on both
subintervals.  The endpoint traces vanish.  Indeed,
$\int_0^{R_\mu/2}r^{-2}|f|^2\,dr<\infty$ forces $f(0+)=0$.  At the
vacuum boundary,
\[
 \|X\|_{Y_\mu}^2
 =4\pi\int_0^{R_\mu}\frac{|f(r)|^2}{\rho_\mu(r)r^2}\,dr<\infty.
\]
Since $\rho_\mu(r)\asymp(R_\mu-r)^{1/(\gamma_0-1)}$ and
$1/(\gamma_0-1)>1$, a nonzero terminal trace would make this integral
diverge; hence $f(R_\mu-)=0$.  We now apply the left-endpoint inequality
on $[0,R_\mu/2]$ and the reflected right-endpoint inequality on
$[R_\mu/2,R_\mu]$ to prove
\eqref{X Y_mu less than BAX X_mu updated}.  First,
\begin{equation}\label{Hardy app 1}
\begin{aligned}
   &\int_{\frac{R_\mu}2}^{R_\mu} \Phi^{\prime\prime}(\rho_\mu) \left|\frac{1}{r^2} \partial_r (r^2 \rho_\mu v)\right|^2 r^2dr\\
   (\Phi^{\prime\prime}(\rho_\mu)\asymp(R_\mu-r)^{\frac{\gamma_0-2}{\gamma_0-1}})
   \gtrsim &   \int_{\frac{R_\mu}2}^{R_\mu}(R_\mu-r)^{\frac{\gamma_0-2}{\gamma_0-1}} \left|\partial_r (r^2 \rho_\mu v)\right|^2dr\\
 \bigl(f(R_\mu-)=0,\ \eqref{Hardy-right-endpoint}\bigr) \gtrsim  &\int_{\frac{R_\mu}2}^{R_\mu}(R_\mu-r)^{\frac{\gamma_0-2}{\gamma_0-1}-2} \left| r^2 \rho_\mu v\right|^2dr\\
(\rho_\mu\asymp(R_\mu-r)^{\frac{1}{\gamma_0-1}}) \gtrsim& \int_{\frac{R_\mu}2}^{R_\mu}\rho_\mu ^{2-\gamma_0}r^4 |v|^2dr\\
\gtrsim &\int_{\frac{R_\mu}2}^{R_\mu} \rho_\mu^{2-\gamma_0} |v|^2 r^2dr.
\end{aligned}
\end{equation}
On the other hand, using the fact that both $\rho_\mu$ and
$\Phi^{\prime\prime}(\rho_\mu)$ are bounded from above and away from
zero on $[0,R_\mu/2]$, and using $f(0+)=0$ in
\eqref{Hardy's ineq 2}, we obtain
\begin{equation} \label{Hardy app 2}
    \begin{aligned}
         &\int^{\frac{R_\mu}2}_{0} \Phi^{\prime\prime}(\rho_\mu) \left|\frac{1}{r^2} \partial_r (r^2 \rho_\mu v)\right|^2 r^2dr
         \gtrsim \int^{\frac{R_\mu}2}_{0}  r^{-2} \left|\partial_r (r^2 \rho_\mu v)\right|^2dr\\
        (\ref{Hardy's ineq 2}) \gtrsim &  \int^{\frac{R_\mu}2}_{0} |\rho_\mu v|^2dr\gtrsim \int^{\frac{R_\mu}2}_{0}  \rho_\mu^{2-\gamma_0} |v|^2 r^2dr.
    \end{aligned}
\end{equation}
Combining \eqref{Hardy app 1} and \eqref{Hardy app 2}, we obtain 
\[
 \int_{S_\mu} \rho_\mu^{2-\gamma_0} |X|^2dx \leq C\|B_\mu A_\mu X\|_{X_\mu}^2.
\]
To prove the second assertion, note that $\gamma_0>1$ and $\rho_\mu$ is bounded, so
\[
 \rho_\mu=\rho_\mu^{2-\gamma_0}\rho_\mu^{\gamma_0-1}
 \le C\rho_\mu^{2-\gamma_0}.
\]
Therefore
\[
 \|X\|_{Y_\mu}^2
 =\int_{S_\mu}\rho_\mu|X|^2\,dx
 \le C\int_{S_\mu}\rho_\mu^{2-\gamma_0}|X|^2\,dx
 \le C\|B_\mu A_\mu X\|_{X_\mu}^2,
\]
which proves \eqref{X-Y-from-BAX}.
\end{proof}
\begin{lemma}\label{Newtonian-potential-bound}
There exists $C>0$ such that, for every $X\in Y_\mu$,
\begin{equation}\label{technical lemma 1}
    \int_{\mathbb{R}^3} |\nabla \phi|^2 \leq C \int_{S_\mu} \rho_\mu |X|^2,
\end{equation}
where $\phi$ satisfies the relation $\Delta \phi= 4\pi \nabla \cdot(\rho_\mu X)$. 
\end{lemma}
\begin{proof}
Extend $\rho_\mu X$ by zero outside $S_\mu$.  Up to the harmless factor
$4\pi$,
\[
 \nabla\phi=\nabla(-\Delta)^{-1}\operatorname{div}(\rho_\mu X).
\]
The Fourier multiplier of
$\nabla(-\Delta)^{-1}\operatorname{div}$ is the orthogonal projection
$\xi\otimes\xi/|\xi|^2$, whose $L^2$ operator norm is one.  Therefore
\[
 \|\nabla\phi\|_{L^2(\mathbb R^3)}
 \le C\|\rho_\mu X\|_{L^2(\mathbb R^3)}
 \le C\|X\|_{Y_\mu},
\]
where boundedness of $\rho_\mu$ was used in the last inequality.
\end{proof}

\begin{lemma}\label{weighted-density-divergence-bound}
There exists $C>0$ such that, for every $X\in Z_\mu$,
\begin{equation}\label{Technical lemma 2}
    \int_{S_\mu} \Phi^{\prime\prime}(\rho_\mu) |\nabla \rho_\mu \cdot X|^2dx +\int_{S_\mu} \Phi^{\prime\prime}(\rho_\mu)|\rho_\mu \nabla \cdot X|^2dx \leq C \| X\|_{Z_\mu}^2.
\end{equation}

\end{lemma}
\begin{proof}
Write
\[
 A:=\nabla\rho_\mu\cdot X,
 \qquad B:=\rho_\mu\operatorname{div}X,
 \qquad A+B=\operatorname{div}(\rho_\mu X).
\]
The two-sided vacuum asymptotics imply, throughout $S_\mu$ after adjusting a
constant,
\[
 \Phi''(\rho_\mu)|\nabla\rho_\mu|^2
 \le C\rho_\mu^{2-\gamma_0}.
\]
Consequently, Lemma~\ref{X Y_mu less than BAX X_mu} gives
\[
 \int_{S_\mu}\Phi''(\rho_\mu)|A|^2\,dx
 \le C\int_{S_\mu}\rho_\mu^{2-\gamma_0}|X|^2\,dx
 \le C\|X\|_{Z_\mu}^2.
\]
Since $B=(A+B)-A$,
\[
 \int_{S_\mu}\Phi''(\rho_\mu)|B|^2\,dx
 \le2\int_{S_\mu}\Phi''(\rho_\mu)|A+B|^2\,dx
   +2\int_{S_\mu}\Phi''(\rho_\mu)|A|^2\,dx,
\]
which proves the result.
\end{proof}

\begin{lemma}[Weighted radial gradient estimate]\label{weighted-radial-gradient}
There exists $C>0$ such that every radial vector field
$Y(x)=v(r)x/r\in Z_\mu$ satisfies
\begin{equation}\label{weighted-radial-gradient-estimate}
 \int_{S_\mu}\rho_\mu^{\gamma_0}|DY|^2\,dx
 \le C\|Y\|_{Z_\mu}^2.
\end{equation}
\end{lemma}

\begin{proof}
For a radial vector field \(Y(x)=v(r)e_r\), one has
\[
 |DY|^2
 =
 |v'|^2+2\frac{|v|^2}{r^2}
 \le
 C\left(
 |\operatorname{div}Y|^2+\frac{|v|^2}{r^2}
 \right).
\]
Since
\[
 \Phi''(\rho_\mu)\rho_\mu^2
 \asymp
 \rho_\mu^{\gamma_0},
\]
Lemma~\ref{weighted-density-divergence-bound} yields
\[
 \int_{S_\mu}
 \rho_\mu^{\gamma_0}|\operatorname{div}Y|^2\,dx
 \le
 C\int_{S_\mu}
 \Phi''(\rho_\mu)|\rho_\mu\operatorname{div}Y|^2\,dx
 \le C\|Y\|_{Z_\mu}^2.
\]

It remains to control
\[
 \int_{S_\mu}
 \rho_\mu^{\gamma_0}\frac{|v|^2}{r^2}\,dx
 =
 4\pi\int_0^{R_\mu}
 \rho_\mu^{\gamma_0}|v|^2\,dr.
\]
We split the radial interval into
\([0,R_\mu/2]\) and \([R_\mu/2,R_\mu]\).

On \([0,R_\mu/2]\), the density is bounded above and away
from zero. Hence
\[
 \rho_\mu^{\gamma_0}\le C\rho_\mu^2.
\]
The interior Hardy estimate \eqref{Hardy app 2} therefore gives
\[
 \int_0^{R_\mu/2}
 \rho_\mu^{\gamma_0}|v|^2\,dr
 \le
 C\int_0^{R_\mu/2}
 |\rho_\mu v|^2\,dr
 \le C\|Y\|_{Z_\mu}^2.
\]

On \([R_\mu/2,R_\mu]\), the boundary Hardy estimate
\eqref{Hardy app 1} gives
\[
 \|Y\|_{Z_\mu}^2
 \ge
 c\int_{R_\mu/2}^{R_\mu}
 \rho_\mu^{2-\gamma_0}|v|^2r^2\,dr.
\]
Since \(r\ge R_\mu/2\),
\[
 \int_{R_\mu/2}^{R_\mu}
 \rho_\mu^{2-\gamma_0}|v|^2\,dr
 \le C\|Y\|_{Z_\mu}^2.
\]
Moreover, because \(\gamma_0>1\) and \(\rho_\mu\) is bounded,
\[
 \rho_\mu^{\gamma_0}
 =
 \rho_\mu^{2-\gamma_0}\rho_\mu^{2\gamma_0-2}
 \le
 C\rho_\mu^{2-\gamma_0}.
\]
Consequently,
\[
 \int_{R_\mu/2}^{R_\mu}
 \rho_\mu^{\gamma_0}|v|^2\,dr
 \le C\|Y\|_{Z_\mu}^2.
\]
Combining the interior and boundary estimates gives
\[
 \int_0^{R_\mu}
 \rho_\mu^{\gamma_0}|v|^2\,dr
 \le C\|Y\|_{Z_\mu}^2.
\]
Multiplying by the angular factor \(4\pi\), we obtain
\[
 \int_{S_\mu}
 \rho_\mu^{\gamma_0}\frac{|v|^2}{r^2}\,dx
 \le C\|Y\|_{Z_\mu}^2.
\]
Combining this with the divergence estimate proves
\[
 \int_{S_\mu}
 \rho_\mu^{\gamma_0}|DY|^2\,dx
 \le C\|Y\|_{Z_\mu}^2.
\]
\end{proof}

The following estimate is the nonlinear Taylor bound used in both instability
arguments.  Its proof is deferred to Section~\ref{proof-remainder-section}.
\begin{lemma}[Radial nonlinear remainder]\label{Estimation of dH/dX-Lmu}
There exist $\delta_*>0$ and a nondecreasing modulus
$\omega:[0,\delta_*]\to[0,\infty)$, with $\omega(0)=0$, such that for
all radial $X,w\in Z_\mu$ satisfying
$\|X-\mathrm{id}\|_{W^{1,\infty}(S_\mu)}\le\delta_*$,
\begin{equation}\label{radial-nonlinear-remainder-estimate}
 \left|
 \left\langle
 \widetilde L_\mu(X-\mathrm{id})-
 \frac{\delta H}{\delta X}(X),w
 \right\rangle
 \right|
 \le
 \omega\!\left(\|X-\mathrm{id}\|_{W^{1,\infty}}
 \right)
 \|X-\mathrm{id}\|_{Z_\mu}\|w\|_{Z_\mu}.
\end{equation}
\end{lemma}

\begin{remark}
The estimate is proved directly in radial Lagrangian variables.  The internal
energy is written in Piola stress form, and the gravitational energy is reduced
by Newton's shell theorem to a one-dimensional integral over mass shells.  This
avoids a derivative loss and does not require a global Hardy--Littlewood--Sobolev estimate.
Only the continuity at vacuum of
$P'(s)/s^{\gamma_0-1}$, which follows from
\eqref{property of pressure 2}, is used to obtain the modulus $\omega$.
\end{remark}

Before moving on, let $w_1,\ldots,w_n$ be an orthonormal basis of the negative eigenspace of $\widetilde{\mathbb L}_\mu$, repeated according to multiplicity, so that
\[
 \widetilde{\mathbb L}_\mu w_i=-\mu_iw_i,
 \qquad 0<\mu_1\le\cdots\le\mu_n.
\]
The existence and finiteness of this family are recalled in Lemma~\ref{spectrum of tilde mathbb L mu}.

\begin{corollary}[Projected nonlinear remainder]\label{Estimation of F}
For $\widetilde X=X-\mathrm{id}$, define
\begin{equation}\label{projected-remainder-definition}
 \mathcal R_-(X)
 :=\sum_{i=1}^n[\widetilde X,w_i]
 \left\langle
 \widetilde L_\mu\widetilde X-\frac{\delta H}{\delta X}(X),w_i
 \right\rangle.
\end{equation}
Whenever $\mathfrak d(X)\le\delta_*$,
\begin{equation}\label{projected-remainder-estimate}
 |\mathcal R_-(X)|
 \le C\omega(\mathfrak d(X))\|\widetilde X\|_{Z_\mu}^2.
\end{equation}
\end{corollary}
\begin{proof}
By Cauchy's inequality in the finite-dimensional negative subspace and Lemma~\ref{Estimation of dH/dX-Lmu},
\[
 \begin{aligned}
 |\mathcal R_-(X)|
 &\le C\|\widetilde X\|_{Y_\mu}
 \omega(\mathfrak d(X))\|\widetilde X\|_{Z_\mu}\\
 &\le C\omega(\mathfrak d(X))\|\widetilde X\|_{Z_\mu}^2,
 \end{aligned}
\]
where Lemma~\ref{X Y_mu less than BAX X_mu} was used in the last step.
\end{proof}
\begin{corollary} \label{Estimation of F_1}
    Define the potential energy part in $H$ by 
    \[
    F_1(X)=\int_{X(S_\mu)}\Phi(\rho)dx-\frac{1}{8\pi}\int_{\mathbb{R}^3}|\nabla V|^2dx+c.
    \]
Then there exists $C_2>0$ such that, whenever
$\mathfrak d(X)\le\delta_*$,
\begin{equation}\label{potential-energy-Taylor-estimate}
 \left|F_1(X)-\frac12
 \left[\widetilde{\mathbb L}_\mu(X-\mathrm{id}),X-\mathrm{id}\right]\right|
 \le C_2\omega(\mathfrak d(X))\|X-\mathrm{id}\|_{Z_\mu}^2.
\end{equation}
\end{corollary}
\begin{proof}
We write $X_\theta=\mathrm{id}+\theta(X-\mathrm{id})$, where $\theta \in (0,1)$. Then,
\[
\begin{aligned}
    &\left| F_1(X)-\frac{1}{2}\left[ \widetilde{\mathbb L}_\mu (X-\mathrm{id}),(X-\mathrm{id})\right] \right| \\
    =&\left| F_1(X)-F_1(\mathrm{id})-\frac{1}{2}\left[ \widetilde{\mathbb L}_\mu (X-\mathrm{id}),(X-\mathrm{id})\right] \right|\\
    \leq &\int_0^1 \left| \left\langle \widetilde L_\mu (X_\theta-\mathrm{id})- \frac{\delta H}{\delta X}(X_\theta), X-\mathrm{id} \right\rangle \right| d\theta\\
    \leq & C_2\omega(\mathfrak d(X))\big\| X-\mathrm{id}\big\|_{Z_\mu}^2.
\end{aligned}
\]
\end{proof}

\section{Instability of Nonrotating Stars}
\subsection{Properties of the Second Variation of the Hamiltonian}\label{Properties of the Second Variation of the Energy Functional}
Recall the weighted spaces and operators from Subsection~\ref{notation-subsection}.  By Lemma~\ref{second variation of H}, the displacement Hessian at the equilibrium is the form $\widetilde L_\mu$.  We first record the spectral facts from \cite{separable} that will be used in both instability arguments.
\begin{lemma}\label{spectrum of L_mu} $L_\mu$ is bounded and self-dual.
    Define $\mathbb{L}_\mu:=I_{X_\mu}L_\mu=I-\frac{4\pi}{\Phi^{\prime\prime}(\rho_\mu)}(-\Delta)^{-1}$, where $I_{X_\mu}=\frac{1}{\Phi^{\prime\prime}(\rho_\mu)}:X_\mu^*\to X_\mu $ is the isomorphism induced from the Riesz representation. Then, $\mathbb{L}_\mu$ is bounded and self-adjoint on $X_\mu$ and $\mathbb{L}_\mu-I$ is compact. 
\end{lemma}
\begin{proof}
    See Lemma~3.6 in \cite{separable}.
\end{proof}
\begin{lemma} \label{Spectrum of mathbbL_mu}

$(1)$ $\widetilde L_\mu$ is self-dual.\\
$(2)$ $\widetilde{\mathbb L}_\mu:=B_\mu^\prime L_\mu B_\mu A_\mu$ is self-adjoint on the Hilbert space $(Y_\mu, [\cdot, \cdot])$, where $[\cdot,\cdot]:=\langle A_\mu \cdot, \cdot\rangle_{(L^2(\mathbb{R}^3))^3}$ is an equivalent inner product. Note that the latter is indeed an inner product because $\rho_\mu>0$ in $S_\mu$ and only vanishes on its boundary, so $\ker A_\mu=\{0\}$.\\
$(3)$ The spectrum of  $\widetilde{\mathbb L}_\mu$ is purely discrete with finite multiplicity and has no accumulating point except for $+\infty$. Moreover, the eigenfunctions of  $\widetilde{\mathbb L}_\mu$ form an orthonormal basis of $Y_\mu$.
\end{lemma}
\begin{proof}
    See Lemmas~2.9 and 3.22 in \cite{separable}.
\end{proof}

\begin{lemma}\label{spectrum of tilde mathbb L mu}
Assume $M'(\mu)\ne0$.  Then
\[
 \ker\widetilde{\mathbb L}_\mu=\{0\}.
\]
Moreover,
\begin{equation}\label{negative-index-displacement-hessian}
 n^-(\widetilde{\mathbb L}_\mu)
 =n^-\!\left(L_\mu\big|_{R(B_\mu A_\mu)}\right),
\end{equation}
and the condition $n^u(\mu)>0$ implies that this number is positive.
Consequently the negative spectrum of $\widetilde{\mathbb L}_\mu$ consists
of finitely many eigenvalues
\[
 -\mu_n\le\cdots\le-\mu_1<0,
 \qquad 0<\mu_1\le\cdots\le\mu_n.
\]
\end{lemma}
\begin{proof}
We give the kernel argument in detail.  Let
$w\in D(\widetilde{\mathbb L}_\mu)$ satisfy
$\widetilde{\mathbb L}_\mu w=0$, and set
\[
 \sigma:=B_\mu A_\mu w=-\operatorname{div}(\rho_\mu w).
\]
Since
$\widetilde{\mathbb L}_\mu=B_\mu'L_\mu B_\mu A_\mu$, one has
\[
 B_\mu'L_\mu\sigma=\nabla(L_\mu\sigma)=0
\]
in $S_\mu$.  Hence
\begin{equation}\label{kernel-Lsigma-constant}
 L_\mu\sigma=c
\end{equation}
for some constant $c$.  The range of $B_\mu A_\mu$ consists of radial
zero-mass perturbations; in particular,
\begin{equation}\label{sigma-zero-mass}
 \int_{S_\mu}\sigma\,dx=0.
\end{equation}
Differentiating the equilibrium equation along the steady-state family gives
(see \cite[Section~3.5]{separable})
\begin{equation}\label{steady-family-derivative-identity}
 L_\mu\partial_\mu\rho_\mu
 =-\frac d{d\mu}\!\left(\frac{M(\mu)}{R_\mu}\right),
\end{equation}
where the right-hand side is a constant function on $S_\mu$.  Using the
self-duality of $L_\mu$, \eqref{kernel-Lsigma-constant},
\eqref{sigma-zero-mass}, and \eqref{steady-family-derivative-identity}, we
obtain
\[
 cM'(\mu)
 =\langle L_\mu\sigma,\partial_\mu\rho_\mu\rangle
 =\langle\sigma,L_\mu\partial_\mu\rho_\mu\rangle
 =-\frac d{d\mu}\!\left(\frac{M(\mu)}{R_\mu}\right)
  \int_{S_\mu}\sigma\,dx=0.
\]
Since $M'(\mu)\ne0$, it follows that $c=0$, and hence
$\sigma\in\ker L_\mu$.

The kernel description in \cite[Section~3.5]{separable} has two cases.  If
$\frac d{d\mu}(M(\mu)/R_\mu)\ne0$, then $\ker L_\mu=\{0\}$, so
$\sigma=0$.  If
$\frac d{d\mu}(M(\mu)/R_\mu)=0$, then
$\ker L_\mu=\operatorname{span}\{\partial_\mu\rho_\mu\}$.  In that case
$\sigma=\lambda\partial_\mu\rho_\mu$; integrating and using
\eqref{sigma-zero-mass} gives
$0=\lambda M'(\mu)$, hence again $\sigma=0$.

It remains to show that $B_\mu A_\mu w=0$ forces $w=0$.  Write
$w(x)=f(r)e_r$.  Then
\[
 0=\operatorname{div}(\rho_\mu w)
 =\frac1{r^2}\frac d{dr}\bigl(r^2\rho_\mu f\bigr),
\]
so $r^2\rho_\mu(r)f(r)=C$.  If $C\ne0$, then
\[
 \|w\|_{Y_\mu}^2
 =4\pi C^2\int_0^{R_\mu}\frac{dr}{\rho_\mu(r)r^2}=+\infty,
\]
already because $\rho_\mu(0)>0$ and $r^{-2}$ is not integrable at the
origin.  Thus $C=0$ and $w=0$, proving the kernel assertion.

For the index identity, the map
\[
 T:=B_\mu A_\mu:Z_\mu\longrightarrow X_\mu
\]
is injective and bounded below by Lemma~\ref{X Y_mu less than BAX X_mu};
therefore its range is closed.  Moreover,
\[
 [\widetilde{\mathbb L}_\mu w,w]
 =\langle L_\mu Tw,Tw\rangle.
\]
The map $T$ consequently gives a dimension-preserving correspondence
between negative subspaces of $\widetilde{\mathbb L}_\mu$ and negative
subspaces of $L_\mu$ restricted to $R(T)$, which proves
\eqref{negative-index-displacement-hessian}.  The turning-point index theorem
of \cite[Section~3.5]{separable} states that this restricted index is
positive whenever $n^u(\mu)>0$.  The finiteness and discreteness of the
negative spectrum now follow from Lemma~\ref{Spectrum of mathbbL_mu}.
\end{proof}

\begin{lemma}[Absolute-form norm]\label{absolute-form-norm}
Assume $M'(\mu)\ne0$.  The form domain of
$|\widetilde{\mathbb L}_\mu|^{1/2}$ is $Z_\mu$, and there exist constants
$c,C>0$ such that
\begin{equation}\label{absolute-form-equivalence}
 c\|f\|_{Z_\mu}^2
 \le [|\widetilde{\mathbb L}_\mu|f,f]
 \le C\|f\|_{Z_\mu}^2,
 \qquad f\in Z_\mu.
\end{equation}
If
$f=\sum_{i=1}^n a_iw_i+f_+$ is the decomposition into the negative and
positive spectral subspaces, then
\begin{equation}\label{absolute-form-spectral-expansion}
 [|\widetilde{\mathbb L}_\mu|f,f]
 =\sum_{i=1}^n\mu_i|a_i|^2
  +[\widetilde{\mathbb L}_\mu f_+,f_+].
\end{equation}
\end{lemma}
\begin{proof}
The closed quadratic form associated with $\widetilde{\mathbb L}_\mu$ is
\[
 [\widetilde{\mathbb L}_\mu f,f]
 =\|B_\mu A_\mu f\|_{X_\mu}^2
  -\frac1{4\pi}\|\nabla\phi_f\|_{L^2(\mathbb R^3)}^2,
 \qquad \Delta\phi_f=4\pi\operatorname{div}(\rho_\mu f),
\]
with form domain $Z_\mu$.  Lemmas~\ref{Newtonian-potential-bound} and
\ref{X Y_mu less than BAX X_mu} show that the Newtonian term is form-bounded
and that the form domain is exactly $Z_\mu$.  Since
$0\notin\sigma(\widetilde{\mathbb L}_\mu)$, the spectral gap gives
\[
 \|f\|_{Y_\mu}^2\le C[|\widetilde{\mathbb L}_\mu|f,f].
\]
Consequently,
\[
 \begin{aligned}
 \|B_\mu A_\mu f\|_{X_\mu}^2
 &\le |[\widetilde{\mathbb L}_\mu f,f]|
   +C\|f\|_{Y_\mu}^2
 \le C[|\widetilde{\mathbb L}_\mu|f,f],\\
 [|\widetilde{\mathbb L}_\mu|f,f]
 &=[\widetilde{\mathbb L}_\mu f,f]
   +2\sum_{i=1}^n\mu_i|[f,w_i]|^2
 \le C\|f\|_{Z_\mu}^2.
 \end{aligned}
\]
The first line, together with the $Y_\mu$ estimate, proves the lower bound in
\eqref{absolute-form-equivalence}; the second proves the upper bound.  Formula
\eqref{absolute-form-spectral-expansion} is the spectral theorem.
\end{proof}

\subsection{Proof of Theorem~\ref{Main thm}}
\begin{proof}
Write $\widetilde X=X-\mathrm{id}$.  We argue by contradiction and assume on the interval under consideration that
\begin{equation}\label{lower-energy-bootstrap}
 \mathfrak d(X(t))<\delta.
\end{equation}
Fix $0<\delta\le\delta_0$.  By Lemmas~\ref{Spectrum of mathbbL_mu} and \ref{spectrum of tilde mathbb L mu},
\[
 \sigma(\widetilde{\mathbb L}_\mu)
 =\{-\mu_n,\ldots,-\mu_1\}\cup\{\lambda_j:j\ge1\},
 \qquad \lambda_j\ge\lambda_1>0.
\]
Let $w_i$ be orthonormal eigenfunctions corresponding to the negative eigenvalues.  Set
\[
 a_i(t)=[\widetilde X(t),w_i],
 \qquad b_i(t)=\langle U(t),w_i\rangle_{L^2},
 \qquad
 V(t)=\sum_{i=1}^na_i(t)b_i(t).
\]
Since $X_t=\rho_\mu^{-1}U$, one has $a_i'=b_i$.  The Hamiltonian equation and the definition \eqref{projected-remainder-definition} give the identity
\begin{equation}\label{corrected-dot-V}
 \begin{aligned}
 V'(t)
 &=\sum_{i=1}^nb_i^2
   -\sum_{i=1}^na_i\left\langle\frac{\delta H}{\delta X}(X),w_i\right\rangle\\
 &=\sum_{i=1}^nb_i^2+\sum_{i=1}^n\mu_i a_i^2+\mathcal R_-(X).
 \end{aligned}
\end{equation}
In particular, the quadratic unstable contribution is not part of the nonlinear remainder.

Fix $0<\alpha<1/2$, to be chosen as a function of $\delta$, and decompose
\[
 V'=V_1+V_2,
\]
where
\begin{align*}
 V_1&:=\sum_i b_i^2+(1-\alpha)\sum_i\mu_i a_i^2,\\
 V_2&:=\alpha\sum_i\mu_i a_i^2+\mathcal R_-(X).
\end{align*}
Cauchy's inequality gives
\begin{equation}\label{V1-lower-bound}
 V_1\ge2\sqrt{(1-\alpha)\mu_1}\,|V|.
\end{equation}

We next use the lower-energy condition to control $V_2$.  Write
\[
 \widetilde X=\sum_{i=1}^na_iw_i+X_+,
 \qquad
 Q_+(X):=[\widetilde{\mathbb L}_\mu X_+,X_+]\ge0.
\]
By Lemma~\ref{absolute-form-norm},
\begin{equation}\label{spectral-form-norm-equivalence}
 c\|\widetilde X\|_{Z_\mu}^2
 \le \sum_{i=1}^n\mu_i a_i^2+Q_+(X)
 \le C\|\widetilde X\|_{Z_\mu}^2.
\end{equation}
Under \eqref{lower-energy-bootstrap}, Corollaries~\ref{Estimation of F} and \ref{Estimation of F_1} yield
\begin{align}
 |\mathcal R_-(X)|&\le C\omega(\delta)\|\widetilde X\|_{Z_\mu}^2,\label{Rminus-small}\\
 \left|F_1(X)-\frac12[\widetilde{\mathbb L}_\mu\widetilde X,\widetilde X]\right|
 &\le C\omega(\delta)\|\widetilde X\|_{Z_\mu}^2.\label{F1-small}
\end{align}
Here and below, $C$ is independent of $\delta$.  Since
\[
 H_0=\frac12\|U\|_{Y_\mu^*}^2+F_1(X)<0,
\]
we have $F_1(X)<0$, and hence \eqref{F1-small} gives
\begin{equation}\label{positive-part-controlled}
 Q_+(X)
 \le \sum_i\mu_i a_i^2+C\omega(\delta)\|\widetilde X\|_{Z_\mu}^2.
\end{equation}
Combining \eqref{spectral-form-norm-equivalence} and \eqref{positive-part-controlled}, and then taking $\delta$ sufficiently small, gives
\begin{equation}\label{negative-part-controls-Z}
 \|\widetilde X\|_{Z_\mu}^2\le C\sum_i\mu_i a_i^2.
\end{equation}
Let $C_0$ be the constant in \eqref{negative-part-controls-Z} and let $C_1$ be the constant in \eqref{Rminus-small}.  Set
\begin{equation}\label{alpha-choice}
 \alpha=4C_0C_1\omega(\delta).
\end{equation}
Shrinking $\delta$ once more ensures $0<\alpha<1/2$, and \eqref{Rminus-small} then implies
\begin{equation}\label{V2-coercive}
 V_2\ge c\alpha\|\widetilde X\|_{Z_\mu}^2.
\end{equation}
Consequently,
\begin{equation}\label{V-master-inequality}
 V'\ge a_\delta|V|+c\alpha\|\widetilde X\|_{Z_\mu}^2,
 \qquad
 a_\delta:=2\sqrt{(1-\alpha)\mu_1}.
\end{equation}

The conserved negative Hamiltonian also gives
\begin{equation}\label{H-controlled-by-X}
 |H_0|\le C\|\widetilde X(t)\|_{Z_\mu}^2,
 \qquad
 \|U(t)\|_{Y_\mu^*}^2\le C\|\widetilde X(t)\|_{Z_\mu}^2.
\end{equation}
Indeed, the first inequality follows from \eqref{F1-small} and boundedness of the quadratic form, while the second follows from
$\frac12\|U\|_{Y_\mu^*}^2<-F_1(X)$.  Thus
\begin{equation}\label{V-forcing-by-energy}
 V'\ge a_\delta|V|+c\alpha|H_0|.
\end{equation}
Finally, the two-sided vacuum asymptotics and
$\|\widetilde X\|_{W^{1,\infty}}<\delta$ give
\[
 \|\widetilde X\|_{Z_\mu}^2
 \le C\delta^2\int_{S_\mu}
 \bigl[\rho_\mu+\Phi''(\rho_\mu)
 (|\nabla\rho_\mu|^2+\rho_\mu^2)\bigr]\,dx
 \le C\delta^2.
\]
Together with \eqref{H-controlled-by-X}, this implies
\begin{equation}\label{V-bootstrap-upper}
 \|\widetilde X(t)\|_{Z_\mu}\le C\delta,
 \qquad
 |V(t)|\le C\delta^2.
\end{equation}

Define
\begin{equation}\label{explicit-h-definition}
 h(\delta):=\frac1{\sqrt{1-\alpha(\delta)}}.
\end{equation}
By \eqref{alpha-choice}, $h(\delta)\to1$ as $\delta\to0+$ and
$a_\delta=2\sqrt{\mu_1}/h(\delta)$.  In the proof of
Lemma~\ref{Estimation of dH/dX-Lmu} we take
$\omega(\delta)=C(\omega_P(\delta)+\delta)$; hence
$\alpha(\delta)\ge c\delta$ after adjusting the constants.  In particular,
$\delta^2/\alpha(\delta)\le C\delta$, which is used in the logarithmic
time bounds below.

If $V(0)\ge0$, then \eqref{V-forcing-by-energy} implies
\[
 V(t)\ge \frac{c\alpha|H_0|}{a_\delta}
 \left(e^{a_\delta t}-1\right).
\]
Since
\[
 \frac d{dt}\sum_i a_i(t)^2=2V(t),
\]
integration gives
\begin{equation}\label{Y-growth-proof-formula}
 \|\widetilde X(t)\|_{Y_\mu}^2
 \ge C\alpha|H_0|
 \left(e^{a_\delta t}-1-a_\delta t\right),
 \qquad 0\le t\le T_\delta^*.
\end{equation}
For each fixed $\delta$, the factor $\alpha/a_\delta^2$ is absorbed into the constant $C_\delta$ in \eqref{growth of Y_mu norm}.  Comparing the corresponding lower bound for $V(t)$ with \eqref{V-bootstrap-upper} shows that the bootstrap cannot persist beyond
\[
 \frac{h(\delta)}{2\sqrt{\mu_1}}
 \log\!\left(\frac{C\delta}{|H_0|}\right).
\]

If $V(0)<0$, then while $V<0$, \eqref{V-forcing-by-energy} gives
\[
 V'+a_\delta V\ge c\alpha|H_0|.
\]
Using $V(0)\ge-C\delta^2$, one sees that either the bootstrap fails or $V$ reaches zero no later than
\[
 t_0\le\frac{h(\delta)}{2\sqrt{\mu_1}}
 \log\!\left(\frac{C\delta}{|H_0|}\right).
\]
Starting at that time and applying the preceding $V\ge0$ argument requires at most the same additional time.  Hence in both cases the bootstrap fails by
\[
 T\le\frac{h(\delta)}{\sqrt{\mu_1}}
 \log\!\left(\frac{D\delta}{|H_0|}\right).
\]
Since $0<\delta\le\delta_0$ was arbitrary, this proves \eqref{conditional-lower-energy-escape}.  The explicit norm bound is asserted only in the case $V(0)\ge0$, as stated in the theorem.
\end{proof}

\subsection{Proof of Theorem~\ref{Main result when energy is higher}}
We first record the correct spectral decomposition of the linearized Hamiltonian flow.  No orthogonality of the growing and decaying eigenvectors is asserted.

\begin{lemma}[Riesz projections and exponential trichotomy]\label{linear-trichotomy-lemma}
Let $\mathcal A$ be defined by \eqref{linearized-generator-definition}, and assume $M'(\mu)\ne0$.  Then $\mathcal A$ generates a strongly continuous group on $\vec Z_\mu$.  There is a bounded invariant decomposition
\begin{equation}\label{trichotomy-decomposition}
 \vec Z_\mu=E^u\oplus E^c\oplus E^s,
 \qquad P_u+P_c+P_s=I,
\end{equation}
where $P_u$ is the Riesz projection \eqref{Riesz-projection-definition}, $\dim E^u=\dim E^s=n$, and the spectrum on $E^c$ is purely imaginary.  There is an equivalent Hilbert norm $\|\cdot\|_*$ on $\vec Z_\mu$ such that, with $\lambda_u:=\sqrt{\mu_1}$,
\begin{align}
 \|e^{t\mathcal A}z_u\|_*&\ge e^{\lambda_u t}\|z_u\|_* &&(t\ge0,\,z_u\in E^u),\label{unstable-trichotomy-estimate}\\
 \|e^{t\mathcal A}z_s\|_*&\le e^{-\lambda_u t}\|z_s\|_* &&(t\ge0,\,z_s\in E^s),\label{stable-trichotomy-estimate}\\
 \|e^{t\mathcal A}z_c\|_*&=\|z_c\|_* &&(t\in\mathbb R,\,z_c\in E^c).\label{center-trichotomy-estimate}
\end{align}
\end{lemma}
\begin{proof}
Use the Riesz isometry $Y_\mu^*\ni g\mapsto \rho_\mu^{-1}g\in
Y_\mu$ and write the phase variable as $(f,v)$, where $v=\rho_\mu^{-1}g$.
In these coordinates the generator is the standard wave operator
\[
 \mathcal G=\begin{pmatrix}0&I\\-\widetilde{\mathbb L}_\mu&0\end{pmatrix}
 \quad\hbox{on } Z_\mu\times Y_\mu,
 \qquad
 D(\mathcal G)=D(\widetilde{\mathbb L}_\mu)\times Z_\mu.
\]
Because $\widetilde{\mathbb L}_\mu$ is self-adjoint and its form domain is
$Z_\mu$, the spectral theorem decomposes $Y_\mu$ into its finite-dimensional
negative spectral space and its positive spectral space.  On every negative
eigenblock the wave equation is a two-dimensional hyperbolic system, while on
the positive spectral space it is the standard conservative wave equation.
The direct sum of these evolutions defines a strongly continuous group on
$Z_\mu\times Y_\mu$.  Conjugating by the Riesz isometry gives the group
generated by $\mathcal A$ on $\vec Z_\mu$.  This is also the
separable-Hamiltonian trichotomy of
\cite{lin_2021_instability,separable}.

Let $\widetilde{\mathbb L}_\mu w_i=-\mu_iw_i$.  On the corresponding
two-dimensional hyperbolic block, write
\[
 e_i^\pm:=\binom{w_i}{\pm\sqrt{\mu_i}\rho_\mu w_i},
 \qquad
 c_i^\pm(f,g):=\frac12\left([f,w_i]
 \pm\frac{\langle g,w_i\rangle}{\sqrt{\mu_i}}\right).
\]
After complexifying the phase space in order to define the contour integral,
and then restricting the conjugation-invariant projections to the real phase
space, one has
\[
 P_uz=\sum_i c_i^+(z)e_i^+,
 \qquad
 P_sz=\sum_i c_i^-(z)e_i^-.
\]
These are precisely the finite-rank Riesz projections onto the positive and
negative real spectrum.  In particular, they are bounded.  No orthogonality
between $e_i^+$ and $e_i^-$ is required.

On the positive spectral subspace of $\widetilde{\mathbb L}_\mu$, define
\[
 \|(f,g)\|_c^2
 :=[\widetilde{\mathbb L}_\mu f,f]+\|g\|_{Y_\mu^*}^2.
\]
Lemma~\ref{absolute-form-norm} and the positive spectral gap show that this is
equivalent to the restriction of the $\vec Z_\mu$ norm.  Set
\[
 \|z\|_*^2:=\sum_{i=1}^n
 \bigl(|c_i^+(z)|^2+|c_i^-(z)|^2\bigr)+\|P_cz\|_c^2.
\]
This is an equivalent Hilbert norm on $\vec Z_\mu$.  The hyperbolic
coordinates satisfy
\[
 (c_i^+)'=\sqrt{\mu_i}\,c_i^+,
 \qquad
 (c_i^-)'=-\sqrt{\mu_i}\,c_i^-.
\]
On the center space, write $L_+:=\widetilde{\mathbb L}_\mu|_{Y_+}>0$.
The change of variables $(f,v)\mapsto(L_+^{1/2}f,v)$ conjugates the center
generator to
\[
 \begin{pmatrix}0&L_+^{1/2}\\-L_+^{1/2}&0\end{pmatrix},
\]
which is skew-adjoint on $Y_\mu\times Y_\mu$.  Hence its group is unitary.
Since $g=\rho_\mu v$ and the Riesz map is an isometry,
$\|g\|_{Y_\mu^*}=\|v\|_{Y_\mu}$; consequently the center flow preserves
$[L_+f,f]+\|g\|_{Y_\mu^*}^2=\|(f,g)\|_c^2$.
The estimates
\eqref{unstable-trichotomy-estimate}--\eqref{center-trichotomy-estimate}
follow.
\end{proof}

We also need a nonlinear robustness estimate for the hyperbolic splitting.  The key point is that Lemma~\ref{Estimation of dH/dX-Lmu} controls the forcing only as a form on $Z_\mu$.  We therefore estimate the unstable and stable coordinates directly and use conservation of the full Hamiltonian to control the center component.

\begin{lemma}[Nonlinear unstable cone]\label{nonlinear-invariant-cone-lemma}
Let $Z(t)=(\widetilde X(t),U(t))$ be a strong solution satisfying
\begin{equation}\label{projection-bootstrap}
 \mathfrak d(X(t))<\delta
\end{equation}
on $[0,T]$.  With the adapted norm in
Lemma~\ref{linear-trichotomy-lemma}, set
\[
 p(t):=\|P_uZ(t)\|_*,\qquad
 s(t):=\|P_sZ(t)\|_*,\qquad
 q(t):=\|P_cZ(t)\|_c .
\]
For a continuous scalar function $f$, write
\[
 D_-f(t):=\liminf_{h\downarrow0}\frac{f(t+h)-f(t)}h,
 \qquad
 D^+f(t):=\limsup_{h\downarrow0}\frac{f(t+h)-f(t)}h.
\]
There is a modulus $\eta(\delta)\to0$ as $\delta\to0+$ such that
\begin{align}
 D_-p&\ge \lambda_up-C\eta(\delta)(p+s+q),\label{p-Dini}\\
 D^+s&\le-\lambda_us+C\eta(\delta)(p+s+q),\label{s-Dini}
\end{align}
and
\begin{equation}\label{center-controlled-by-energy}
 q(t)^2\le C\Bigl(\|Z(0)\|_*^2+p(t)s(t)
       +\eta(\delta)\bigl(p(t)^2+s(t)^2\bigr)\Bigr).
\end{equation}
Consequently, let $0<\kappa\le1$.  If
\begin{equation}\label{cone-initial-data-star}
 p(0)\ge \kappa\|Z(0)\|_*
\end{equation}
and
\begin{equation}\label{cone-smallness-condition}
 C\eta(\delta)\kappa^{-1}\le \frac{\lambda_u}{32},
\end{equation}
then, as long as \eqref{projection-bootstrap} holds,
\begin{equation}\label{cone-growth-conclusion}
 s(t)+q(t)\le C\kappa^{-1}p(t),
 \qquad
 p(t)\ge p(0)e^{\lambda_{\kappa,\delta}t},
\end{equation}
where
\begin{equation}\label{effective-cone-growth-rate}
 \lambda_{\kappa,\delta}
 :=\lambda_u-C_\kappa\eta(\delta)>0,
 \qquad
 C_\kappa\le C(1+\kappa^{-1}).
\end{equation}
In particular, for fixed $\kappa$,
$\lambda_{\kappa,\delta}\to\lambda_u$ as $\delta\to0+$.
\end{lemma}
\begin{proof}
Write the negative-mode coordinates as in
Lemma~\ref{linear-trichotomy-lemma}.  If
\[
 r_i(t):=\left\langle
 \widetilde L_\mu\widetilde X(t)-\frac{\delta H}{\delta X}(X(t)),w_i
 \right\rangle,
\]
then the Hamiltonian equations give
\begin{equation}\label{hyperbolic-coordinate-equations}
 (c_i^+)'=\sqrt{\mu_i}\,c_i^++\frac{r_i}{2\sqrt{\mu_i}},
 \qquad
 (c_i^-)'=-\sqrt{\mu_i}\,c_i^- -\frac{r_i}{2\sqrt{\mu_i}}.
\end{equation}
Lemma~\ref{Estimation of dH/dX-Lmu}, finite dimensionality of the negative
space, and equivalence of the adapted and original norms imply
\begin{equation}\label{modal-force-bound}
 \left(\sum_i|r_i|^2\right)^{1/2}
 \le C\eta(\delta)\|\widetilde X\|_{Z_\mu}
 \le C\eta(\delta)(p+s+q).
\end{equation}
Equations \eqref{p-Dini}--\eqref{s-Dini} follow from
\eqref{hyperbolic-coordinate-equations} by the standard one-sided
Dini-derivative inequalities for Euclidean norms.  This formulation also
covers instants at which $p$ or $s$ vanishes.

It remains to control the center component without estimating the nonlinear
force in the whole phase-space norm.  Decompose
\[
 \widetilde X=\sum_i a_iw_i+X_c,
 \qquad
 U=\sum_i b_i\rho_\mu w_i+U_c,
\]
where $X_c$ belongs to the positive spectral subspace.  By the definitions
of $c_i^\pm$,
\[
 \frac12\bigl(b_i^2-\mu_i a_i^2\bigr)
 =-2\mu_i c_i^+c_i^-.
\]
Corollary~\ref{Estimation of F_1} and conservation of the Hamiltonian give
\begin{equation}\label{energy-hyperbolic-decomposition}
 H(Z(0))
 =\frac12q(t)^2-2\sum_i\mu_i c_i^+(t)c_i^-(t)+\mathcal R_E(t),
 \qquad
 |\mathcal R_E(t)|\le\eta(\delta)\|\widetilde X(t)\|_{Z_\mu}^2.
\end{equation}
For sufficiently small initial data, the same expansion at $t=0$ yields
$|H(Z(0))|\le C\|Z(0)\|_*^2$.  Since the negative space is finite
dimensional and the adapted norm is equivalent to the original phase norm,
\[
 \|\widetilde X\|_{Z_\mu}^2\le C(p^2+s^2+q^2),
 \qquad
 \left|\sum_i\mu_i c_i^+c_i^-\right|\le Cps.
\]
Substitution into \eqref{energy-hyperbolic-decomposition} gives
\[
 q^2\le C\bigl(\|Z(0)\|_*^2+ps
 +\eta(\delta)(p^2+s^2+q^2)\bigr).
\]
After shrinking $\delta$ so that the last $q^2$ term is absorbed, this is
\eqref{center-controlled-by-energy}.

Choose $A=A_0\kappa^{-1}$ with $A_0$ large enough that
$s(0)\le Ap(0)$.  Suppose, on a maximal interval, that $s\le Ap$ and
$p\ge p(0)$.  From \eqref{center-controlled-by-energy} and
\eqref{cone-initial-data-star},
\begin{equation}\label{center-inside-cone}
 q\le C\kappa^{-1}p
\end{equation}
provided \eqref{cone-smallness-condition} holds.  Substitution into
\eqref{p-Dini} also gives $D_-p>0$ whenever $p=p(0)$, so the second boundary
of the bootstrap region cannot be crossed.  At the boundary $s=Ap$,
equations \eqref{p-Dini}, \eqref{s-Dini}, and
\eqref{center-inside-cone} give
\[
 D^+(s-Ap)
 \le-2\lambda_uAp+C\eta(\delta)\kappa^{-2}p<0.
\]
Indeed, the negative term has size $\lambda_u\kappa^{-1}p$, whereas the
error has size $\eta(\delta)\kappa^{-2}p$; condition
\eqref{cone-smallness-condition}, after enlarging its harmless constant,
makes the latter strictly smaller.  Thus the cone is invariant.

Within the invariant cone,
\eqref{p-Dini} and \eqref{center-inside-cone} yield the sharper estimate
\[
 D_-p\ge\bigl(\lambda_u-C_\kappa\eta(\delta)\bigr)p.
\]
After decreasing the admissible $\delta$ if necessary, the coefficient is
positive.  Integrating the Dini inequality proves
\eqref{cone-growth-conclusion}--\eqref{effective-cone-growth-rate}; the
bound for $s+q$ follows from $s\le Ap$ and
\eqref{center-inside-cone}.
\end{proof}

\begin{proof}[Proof of Theorem~\ref{Main result when energy is higher}]
Fix $\kappa\in(0,1]$.  By equivalence of the original and adapted norms,
there is a constant $c_*\in(0,1]$, independent of $\kappa$ and of the
initial data, such that
\[
 \|P_uZ\|_*\ge c_*\|P_uZ\|_{\vec Z_\mu},
 \qquad
 \|Z\|_{\vec Z_\mu}\ge c_*\|Z\|_*.
\]
Consequently, \eqref{unstable-cone-initial-condition} implies
\begin{equation}\label{adapted-initial-cone-condition}
 p(0)\ge c_*^2\kappa\|Z(0)\|_*.
\end{equation}
Set $\kappa_*:=c_*^2\kappa$.  Since $\eta(\delta)\to0$, choose
$\bar\delta_\kappa>0$ so that
\begin{equation}\label{kappa-dependent-bootstrap-smallness}
 C\eta(\delta)\kappa_*^{-1}\le\frac{\lambda_u}{32}
 \qquad\text{for }0<\delta\le\bar\delta_\kappa.
\end{equation}
Fix such a $\delta$.  Lemma~\ref{nonlinear-invariant-cone-lemma} applies
with cone aperture $\kappa_*$.

Assume, toward a contradiction, that
\eqref{conditional-projection-escape} fails.  Then
\eqref{projection-bootstrap} holds with this $\delta$ throughout the
interval.  Energy conservation and Corollary~\ref{Estimation of F_1} imply
\begin{equation}\label{phase-norm-bootstrap-bound}
 \|Z(t)\|_*\le C\delta.
\end{equation}
Indeed, for a radial displacement on the fixed ball, the definition of
$Z_\mu$ and the two-sided vacuum asymptotics give
\[
 \|\widetilde X(t)\|_{Z_\mu}
 \le C\|\widetilde X(t)\|_{W^{1,\infty}(S_\mu)}
 =C\mathfrak d(X(t)).
\]
Moreover,
\[
 \frac12\|U(t)\|_{Y_\mu^*}^2
 =H(Z(0))-F_1(X(t))
\]
is bounded by
$C\|Z(0)\|_{\vec Z_\mu}^2+C\|\widetilde X(t)\|_{Z_\mu}^2$ through
Corollary~\ref{Estimation of F_1}.

By \eqref{adapted-initial-cone-condition} and
Lemma~\ref{nonlinear-invariant-cone-lemma},
\[
 p(t)\ge c_*^2\kappa\|Z(0)\|_*
 e^{\lambda_{\kappa_*,\delta}t}.
\]
This contradicts \eqref{phase-norm-bootstrap-bound} once
\[
 t=\frac1{\lambda_{\kappa_*,\delta}}
 \log\!\left(
 \frac{2C\delta}{c_*^2\kappa\|Z(0)\|_*}
 \right).
\]
After one final use of norm equivalence, the theorem follows with
\begin{equation}\label{explicit-h-kappa}
 h_\kappa(\delta)
 :=\frac{\sqrt{\mu_1}}{\lambda_{\kappa_*,\delta}}
 =\frac{\sqrt{\mu_1}}
 {\sqrt{\mu_1}-C_{\kappa_*}\eta(\delta)},
 \qquad
 D_2(\kappa,\delta)=\frac{C\delta}{\kappa}.
\end{equation}
For fixed $\kappa$, $\eta(\delta)\to0$ implies
$h_\kappa(\delta)\to1$, so the time prefactor tends to
$1/\sqrt{\mu_1}$ as claimed.
\end{proof}

\subsection{Proof of Theorem~\ref{Main thm 0}}\label{proof-polytropic-instability-subsection}
\begin{proof}
The Lane--Emden scaling gives
\[
 R_\mu=R_1\mu^{(\gamma-2)/2},
 \qquad
 M(\mu)=M(1)\mu^{(3\gamma-4)/2}.
\]
Thus, for $6/5<\gamma<4/3$, one has $M'(\mu)<0$ and
$\frac d{d\mu}(M(\mu)/R_\mu)>0$, so $i_\mu=0$.  Scaling also conjugates
the radial linearized operators at different central densities up to a
positive factor, hence preserves their negative index.  The small-density
value in Theorem~\ref{turning pt principle} therefore yields
$n^u(\mu)=1$ for every $\mu>0$.

Fix the escape radius $\delta$ in the relevant conditional theorem (with
$\delta\le\bar\delta_\kappa$ in the cone case).  By
Theorem~\ref{Well-posedness} and
\eqref{high-energy-controls-bootstrap}, choose $e_*>0$ so small that
\begin{equation}\label{energy-implies-bootstrap}
 \mathcal N_\gamma(t)<e_*
 \quad\Longrightarrow\quad
 \mathfrak d(X(t))<\delta,
\end{equation}
and the Jacobian and physical-vacuum constants remain in a fixed restart
regime.  Set
\[
 \nu_*:=\min\left\{e_*,\frac{\delta}{2C}\right\},
\]
where $C$ is the constant in
\eqref{high-energy-controls-bootstrap}.  In the cone case these constants
may depend on $\kappa$.

Choose smooth compatible initial data of the appropriate type, as described
after Theorem~\ref{Main thm 0}, with
$0<\mathcal N_\gamma(0)<\min\{e_*/2,\varepsilon_*\}$.  Let
$[0,T_{\max})$ be the maximal classical existence interval and define
\[
 \tau_*:=\sup\left\{t<T_{\max}:\
 \mathcal N_\gamma(s)<e_*\ \text{for }0\le s<t\right\}.
\]
Let $T_{\mathrm{esc}}$ denote the logarithmic time in the corresponding
conditional theorem.

If $\tau_*\le T_{\mathrm{esc}}$, then either
$\mathcal N_\gamma$ reaches $e_*$ by time $\tau_*$, which already gives the
desired escape, or $\tau_*=T_{\max}$ while all continuation quantities stay
in the controlled regime.  The latter contradicts the restart property in
Theorem~\ref{Well-posedness}.

If $\tau_*>T_{\mathrm{esc}}$, the same restart property and uniqueness give
one classical solution on the whole interval $[0,T_{\mathrm{esc}}]$, and
\eqref{energy-implies-bootstrap} keeps it inside the deformation bootstrap
there.  Theorem~\ref{Main thm} or
Theorem~\ref{Main result when energy is higher} then gives a time
$t_*\le T_{\mathrm{esc}}$ for which $\mathfrak d(X(t_*))\ge\delta$.  Hence,
by \eqref{bootstrap-forces-N-gamma},
\[
 \mathcal N_\gamma(t_*)\ge C^{-1}\delta\ge2\nu_*.
\]
Taking $\nu_{\mathrm{low}}=\nu_*$ in part~(i), and the corresponding
$\nu_{\mathrm{cone}}(\kappa)=\nu_*$ in part~(ii), proves the theorem.
\end{proof}

\section{Proof of Lemma~\ref{Estimation of dH/dX-Lmu}}\label{proof-remainder-section}
\begin{proof}
Put $Y=X-\mathrm{id}$.  Since all fields are radial, write
\[
 X(x)=\chi(r)e_r,
 \qquad
 Y(x)=\xi(r)e_r,
 \qquad
 w(x)=\psi(r)e_r,
 \qquad e_r=\frac{x}{r}.
\]
For $\|Y\|_{W^{1,\infty}}$ sufficiently small, $\chi$ is strictly
increasing and
\begin{equation}\label{radial-deformation-parameters}
 a:=\chi'=1+\xi',
 \qquad
 b:=\frac{\chi}{r}=1+\frac{\xi}{r},
 \qquad
 J:=\det DX=ab^2
\end{equation}
remain in a fixed compact subset of $(0,\infty)$.  A regular radial
Lipschitz map fixes the origin, and therefore
\begin{equation}\label{radial-Lipschitz-profile-control}
 \|\xi'\|_{L^\infty(0,R_\mu)}
 +\|\xi/r\|_{L^\infty(0,R_\mu)}
 \le C\|Y\|_{W^{1,\infty}(S_\mu)}.
\end{equation}
Let
\[
 m_\mu(r):=4\pi\int_0^r\rho_\mu(s)s^2\,ds
\]
be the mass enclosed by the reference shell of radius $r$.

\medskip
\noindent\emph{Internal energy and the Piola stress.}
In radial variables,
\[
 \mathcal U(X)=4\pi\int_0^{R_\mu}
 \Phi\!\left(\frac{\rho_\mu}{J}\right)Jr^2\,dr.
\]
We derive its first variation.  Set
$X_\varepsilon=(\chi+\varepsilon\psi)e_r$.  Then
\[
 a_\varepsilon=a+\varepsilon\psi',
 \qquad
 b_\varepsilon=b+\varepsilon\frac{\psi}{r},
 \qquad
 J_\varepsilon=a_\varepsilon b_\varepsilon^2,
\]
so
\begin{equation}\label{radial-Jacobian-variation}
 \left.\frac d{d\varepsilon}J_\varepsilon\right|_{\varepsilon=0}
 =b^2\psi'+2ab\frac{\psi}{r}.
\end{equation}
If $q_\varepsilon=\rho_\mu/J_\varepsilon$, then
\[
 \frac d{d\varepsilon}
 \bigl[\Phi(q_\varepsilon)J_\varepsilon\bigr]
 =\bigl(\Phi(q_\varepsilon)-q_\varepsilon\Phi'(q_\varepsilon)\bigr)
 J_\varepsilon'
 =-P(q_\varepsilon)J_\varepsilon',
\]
where $P(q)=q\Phi'(q)-\Phi(q)$.  Combining this identity with
\eqref{radial-Jacobian-variation} gives the exact radial Piola-stress formula
\begin{equation}\label{radial-internal-first-variation}
 D\mathcal U(X)[w]
 =-4\pi\int_0^{R_\mu}
 P\!\left(\frac{\rho_\mu}{ab^2}\right)
 \left(b^2\psi'+2ab\frac{\psi}{r}\right)r^2\,dr.
\end{equation}
Equivalently, this is the radial form of
$D\mathcal U(X)[w]=-
\int P(\rho_\mu/J)\operatorname{cof}(DX):Dw$.

For $0\le\rho\le\|\rho_\mu\|_\infty$, define the two-component stress
\[
 \mathcal S_\rho(a,b)
 :=P\!\left(\frac{\rho}{ab^2}\right)(b^2,2ab).
\]
Writing $q=\rho/(ab^2)$, direct differentiation gives
\begin{align*}
 \partial_a\mathcal S_{\rho,1}
 &=-\frac{b^2}{a}\,qP'(q),
 &\partial_b\mathcal S_{\rho,1}
 &=2b\bigl(P(q)-qP'(q)\bigr),\\
 \partial_a\mathcal S_{\rho,2}
 &=2b\bigl(P(q)-qP'(q)\bigr),
 &\partial_b\mathcal S_{\rho,2}
 &=2a\bigl(P(q)-2qP'(q)\bigr).
\end{align*}
The normalization $P(0)=0$ and \eqref{property of pressure 2} imply
\[
 \lim_{s\to0+}\frac{P(s)}{s^{\gamma_0}}
 =\frac{C_{\gamma_0}}{\gamma_0},
 \qquad
 \lim_{s\to0+}\frac{sP'(s)}{s^{\gamma_0}}
 =C_{\gamma_0}.
\]
Hence $\rho^{-\gamma_0}D_{a,b}\mathcal S_\rho(a,b)$ extends continuously
to $\rho=0$, uniformly for $(a,b)$ in a fixed compact neighborhood of
$(1,1)$.  By uniform continuity and the integral form of Taylor's theorem,
there is a modulus $\omega_P(\delta)\to0$ such that, whenever
$|a-1|+|b-1|\le\delta$,
\begin{equation}\label{radial-stress-Taylor}
 \left|
 \mathcal S_\rho(a,b)-\mathcal S_\rho(1,1)
 -D\mathcal S_\rho(1,1)(a-1,b-1)
 \right|
 \le \omega_P(\delta)\rho^{\gamma_0}
 \bigl(|a-1|+|b-1|\bigr).
\end{equation}
Indeed, the left-hand side equals
\[
 \int_0^1
 \bigl[D\mathcal S_\rho((1,1)+\theta(a-1,b-1))
       -D\mathcal S_\rho(1,1)\bigr]
 (a-1,b-1)\,d\theta.
\]

Let
$\mathbf h=(\xi',\xi/r)$ and
$\mathbf q=(\psi',\psi/r)$.  Subtracting the value and first derivative of
\eqref{radial-internal-first-variation} at the identity gives
\begin{align*}
 &D\mathcal U(X)[w]-D\mathcal U(\mathrm{id})[w]
 -D^2\mathcal U(\mathrm{id})[Y,w]\\
 &\quad=-4\pi\int_0^{R_\mu}
 \bigl[\mathcal S_{\rho_\mu}(a,b)-\mathcal S_{\rho_\mu}(1,1)
 -D\mathcal S_{\rho_\mu}(1,1)\mathbf h\bigr]
 \cdot\mathbf q\,r^2\,dr.
\end{align*}
Using \eqref{radial-stress-Taylor}, Cauchy--Schwarz, and the identity
$|DY|^2=|\xi'|^2+2|\xi/r|^2$ (and similarly for $w$), we obtain
\begin{align*}
 &\left|D\mathcal U(X)[w]-D\mathcal U(\mathrm{id})[w]
 -D^2\mathcal U(\mathrm{id})[Y,w]\right|\\
 &\quad\le C\omega_P(\|Y\|_{W^{1,\infty}})
 \left(\int_{S_\mu}\rho_\mu^{\gamma_0}|DY|^2\,dx\right)^{1/2}
 \left(\int_{S_\mu}\rho_\mu^{\gamma_0}|Dw|^2\,dx\right)^{1/2}.
\end{align*}
Lemma~\ref{weighted-radial-gradient} therefore yields
\begin{equation}\label{radial-internal-remainder}
 \left|D\mathcal U(X)[w]-D\mathcal U(\mathrm{id})[w]
 -D^2\mathcal U(\mathrm{id})[Y,w]\right|
 \le C\omega_P(\|Y\|_{W^{1,\infty}})
 \|Y\|_{Z_\mu}\|w\|_{Z_\mu}.
\end{equation}

\medskip
\noindent\emph{Gravitational energy and Newton's shell theorem.}
For a spherically symmetric density, the gravitational energy may be built
shell by shell.  The Lagrangian shell labeled by $r$ carries the fixed mass
$dm_\mu(r)=4\pi\rho_\mu(r)r^2dr$ and is located at the physical radius
$\chi(r)$.  Because $\chi$ is increasing, the mass inside that shell is
exactly $m_\mu(r)$.  Newton's shell theorem says that the potential produced
there by the interior mass is $-m_\mu(r)/\chi(r)$.  Adding the shells from
the center outward counts every pair of shells once and gives
\begin{equation}\label{radial-gravitational-energy}
 \mathcal W(X)
 =-\int_0^{R_\mu}\frac{m_\mu(r)}{\chi(r)}\,dm_\mu(r)
 =-4\pi\int_0^{R_\mu}
 \frac{m_\mu(r)\rho_\mu(r)r^2}{\chi(r)}\,dr.
\end{equation}
This is equivalent to
$-\frac12\iint\rho(y)\rho(z)|y-z|^{-1}\,dy\,dz$, or to
$-(8\pi)^{-1}\int|\nabla V|^2$, with the sign convention
$\Delta V=4\pi\rho$.

Differentiating \eqref{radial-gravitational-energy} along
$\chi+\varepsilon\psi$ gives
\begin{equation}\label{radial-gravity-first-variation}
 D\mathcal W(X)[w]
 =4\pi\int_0^{R_\mu}
 m_\mu(r)\rho_\mu(r)r^2
 \frac{\psi(r)}{\chi(r)^2}\,dr.
\end{equation}
Set $z=\xi/r$, so $\chi=r(1+z)$.  At the identity,
\begin{align*}
 D\mathcal W(\mathrm{id})[w]
 &=4\pi\int_0^{R_\mu}m_\mu\rho_\mu\psi\,dr,\\
 D^2\mathcal W(\mathrm{id})[Y,w]
 &=-8\pi\int_0^{R_\mu}m_\mu\rho_\mu z\psi\,dr.
\end{align*}
Thus the remainder is exactly
\[
 4\pi\int_0^{R_\mu}m_\mu\rho_\mu\psi
 \bigl[(1+z)^{-2}-1+2z\bigr]dr.
\]
For $|z|\le\delta<1/2$,
\[
 |(1+z)^{-2}-1+2z|\le C\delta|z|.
\]
Moreover, $m_\mu(r)/r^3$ extends continuously to $r=0$ with value
$4\pi\rho_\mu(0)/3$ and is bounded on $[0,R_\mu]$.  Hence
\begin{align*}
 &\left|D\mathcal W(X)[w]-D\mathcal W(\mathrm{id})[w]
 -D^2\mathcal W(\mathrm{id})[Y,w]\right|\\
 &\quad\le C\|Y\|_{W^{1,\infty}}
 \int_0^{R_\mu}\rho_\mu r^2|\xi||\psi|\,dr\\
 &\quad\le C\|Y\|_{W^{1,\infty}}
 \left(\int_0^{R_\mu}\rho_\mu|\xi|^2r^2\,dr\right)^{1/2}
 \left(\int_0^{R_\mu}\rho_\mu|\psi|^2r^2\,dr\right)^{1/2}.
\end{align*}
Therefore
\begin{equation}\label{radial-gravity-remainder}
 \left|D\mathcal W(X)[w]-D\mathcal W(\mathrm{id})[w]
 -D^2\mathcal W(\mathrm{id})[Y,w]\right|
 \le C\|Y\|_{W^{1,\infty}}
 \|Y\|_{Y_\mu}\|w\|_{Y_\mu}.
\end{equation}

Finally, $(\mathrm{id},0)$ is a critical point of the full Hamiltonian, and
Lemma~\ref{second variation of H} identifies
$D^2(\mathcal U+\mathcal W)(\mathrm{id})$ with
$\widetilde L_\mu$.  Combining
\eqref{radial-internal-remainder},
\eqref{radial-gravity-remainder}, and
Lemma~\ref{X Y_mu less than BAX X_mu}, the conclusion follows with
\[
 \omega(\delta):=C\bigl(\omega_P(\delta)+\delta\bigr).
\]
\end{proof}

\appendix
\section{The Luo--Xin--Zeng high-order energy}\label{appendix-LXZ-energy}
For reference, we record the radial energy from \cite[Section~9]{Xinzhouping}
that underlies Theorem~\ref{Well-posedness}.  After rescaling the reference
ball to $I=(0,1)$, set
\[
 \sigma(x):=x\rho_0(x)^{\gamma-1},\qquad
 \nu:=\frac{2-\gamma}{2\gamma-2},\qquad
 \ell:=3+2\left\lceil\frac12+\nu\right\rceil,
\]
and let $\zeta$ be the cutoff of \cite[Section~9]{Xinzhouping}, equal to one
near $x=0$ and zero away from the center.  With $\|\cdot\|_j$ denoting the
$H^j(I)$ norm, their high-order velocity energy is
{\small
\begin{align}
 \widetilde{\mathcal E}_\gamma(v,t)
 := {}&\left\|\sigma\left(\frac{\sigma}{x}\right)^\nu
       \partial_t^\ell v_x\right\|_0^2
 +\left\|\left(\frac{\sigma}{x}\right)^{1+\nu}
       \partial_t^\ell v\right\|_0^2 \notag\\
 &+\sum_{j=1}^{(\ell+1)/2}\left\{
   \left\|\sigma^{3/2+\nu}
    \partial_t^{\ell-2j+1}\partial_x^{j+1}v\right\|_0^2
   +\sum_{i=0}^{j}
    \left\|\sigma^{1/2+\nu}
    \partial_t^{\ell-2j+1}\partial_x^iv\right\|_0^2
   \right\} \notag\\
 &+\sum_{j=1}^{(\ell-1)/2}\left\{
   \left\|\sigma^{2+\nu}
    \partial_t^{\ell-2j}\partial_x^{j+2}v\right\|_0^2
   +\sum_{i=-1}^{j}
    \left\|\sigma^{1+\nu}
    \partial_t^{\ell-2j}\partial_x^{i+1}v\right\|_0^2
   \right\} \notag\\
 &+\sum_{j=1}^{(\ell+1)/2}\left\{
   \left\|\zeta\sigma\partial_t^{\ell-2j+1}v\right\|_{j+1}^2
   +\left\|\zeta\partial_t^{\ell-2j+1}v\right\|_j^2
   +\left\|\zeta\frac{\partial_t^{\ell-2j+1}v}{x}\right\|_{j-1}^2
   \right\} \notag\\
 &+\sum_{j=1}^{(\ell-1)/2}\left\{
   \left\|\zeta\sigma\partial_t^{\ell-2j}v\right\|_{j+2}^2
   +\left\|\zeta\partial_t^{\ell-2j}v\right\|_{j+1}^2
   +\left\|\zeta\frac{\partial_t^{\ell-2j}v}{x}\right\|_j^2
   \right\}.
 \label{LXZ-high-order-energy}
\end{align}
}
The time derivatives at $t=0$ are generated recursively from the radial
equation.  The density assumptions used in \cite[(3.3)--(3.4)]{Xinzhouping}
are
\[
 \rho_0(x)>0\quad(0\le x<1),\qquad \rho_0(1)=0,
 \qquad -\infty<\partial_x(\rho_0^{\gamma-1})(1)<0,
\]
together with the stated interior regularity.  The local existence argument
uses the estimates surrounding \cite[(9.5), (9.10), (9.11)]{Xinzhouping}
and the approximation in their Subsection~9.4; uniqueness in the range used
here is \cite[Theorem~2.2]{Xinzhouping}.  The uniform restart property in
Theorem~\ref{Well-posedness} is the time-translation and change-of-chart
consequence of reapplying this local construction while these controlling
quantities remain in a fixed bounded set.

For completeness, if the initial radial relabeling is
$X_0(x)=\chi_0(r)e_r$, mass conservation gives the pushed-forward density
\[
 \rho_0^{\rm phys}(\chi_0(r))
 =\frac{\rho_\mu(r)}{\chi_0'(r)(\chi_0(r)/r)^2}.
\]
If $X_0$ is a smooth radial diffeomorphism sufficiently close to the
identity, the denominator is bounded above and below and distance to the
boundary is distorted only by fixed factors; hence the physical-vacuum
behavior is preserved.  The remaining high-order compatibility conditions
are imposed on the smooth data used in Theorem~\ref{Main thm 0}.

\section*{Acknowledgments}

The authors thank Professor Chongchun Zeng for helpful discussions.

\bibliographystyle{amsplain}
\bibliography{reference}

@book{subrahmanijanchandrasekhar_1967_an,
  author    = {Chandrasekhar, Subrahmanyan},
  title     = {An Introduction to the Study of Stellar Structure},
  publisher = {Dover Publications},
  address   = {New York},
  year      = {1967}
}

@incollection{Gidas1981SymmetryOP,
  author    = {Gidas, Basilis and Ni, Wei-Ming and Nirenberg, Louis},
  title     = {Symmetry of Positive Solutions of Nonlinear Elliptic Equations in $\mathbb{R}^n$},
  booktitle = {Mathematical Analysis and Applications, Part A},
  series    = {Advances in Mathematics Supplementary Studies},
  volume    = {7A},
  pages     = {369--402},
  publisher = {Academic Press},
  address   = {New York},
  year      = {1981}
}

@article{heinzle_2003_newtonian,
  author  = {Heinzle, J. Mark and Uggla, Claes},
  title   = {Newtonian Stellar Models},
  journal = {Annals of Physics},
  volume  = {308},
  pages   = {18--61},
  year    = {2003}
}

@article{lin_1997_stability,
  author  = {Lin, Song-Sun},
  title   = {Stability of Gaseous Stars in Spherically Symmetric Motions},
  journal = {SIAM Journal on Mathematical Analysis},
  volume  = {28},
  number  = {3},
  pages   = {539--569},
  year    = {1997}
}

@article{jang_2008_nonlinear,
  author  = {Jang, Juhi},
  title   = {Nonlinear Instability in Gravitational {Euler--Poisson} Systems for $\gamma=6/5$},
  journal = {Archive for Rational Mechanics and Analysis},
  volume  = {188},
  number  = {2},
  pages   = {265--307},
  year    = {2008}
}

@article{Jang_2013_nonlinear,
  author  = {Jang, Juhi},
  title   = {Nonlinear Instability Theory of {Lane--Emden} Stars},
  journal = {Communications on Pure and Applied Mathematics},
  volume  = {67},
  number  = {9},
  pages   = {1418--1465},
  year    = {2014}
}

@article{Collapse_solution,
  author  = {Guo, Yan and Had\v{z}i\'{c}, Mahir and Jang, Juhi and Schrecker, Matthew},
  title   = {Gravitational Collapse for Polytropic Gaseous Stars: Self-Similar Solutions},
  journal = {Archive for Rational Mechanics and Analysis},
  volume  = {246},
  number  = {2--3},
  pages   = {957--1066},
  year    = {2022}
}

@article{Expanding_solution,
  author  = {Cheng, Ming and Cheng, Xing and Lin, Zhiwu},
  title   = {Expanding Solutions Near Unstable {Lane--Emden} Stars},
  journal = {Communications in Mathematical Physics},
  volume  = {406},
  number  = {7},
  pages   = {Article 165},
  year    = {2025},
  doi     = {10.1007/s00220-025-05308-3}
}

@article{separable,
  author  = {Lin, Zhiwu and Zeng, Chongchun},
  title   = {Separable Hamiltonian {PDEs} and Turning Point Principle for Stability of Gaseous Stars},
  journal = {Communications on Pure and Applied Mathematics},
  volume  = {75},
  number  = {11},
  pages   = {2511--2572},
  year    = {2022}
}

@article{lin_2024_nonlinear,
  author  = {Lin, Zhiwu and Wang, Yucong and Zhu, Hao},
  title   = {Nonlinear Stability of Non-Rotating Gaseous Stars},
  journal = {Mathematische Annalen},
  volume  = {391},
  number  = {1},
  pages   = {843--880},
  year    = {2025}
}

@article{rein_2003_nonlinear,
  author  = {Rein, Gerhard},
  title   = {Non-Linear Stability of Gaseous Stars},
  journal = {Archive for Rational Mechanics and Analysis},
  volume  = {168},
  pages   = {115--130},
  year    = {2003}
}

@article{aly_1992_on,
  author  = {Aly, Jean-Jacques and Perez, Jerome},
  title   = {On the Stability of a Gaseous Sphere Against Non-Radial Perturbations},
  journal = {Monthly Notices of the Royal Astronomical Society},
  volume  = {259},
  pages   = {95--103},
  year    = {1992}
}

@inproceedings{Antonov_1987,
  author    = {Antonov, V. A.},
  title     = {Solution of the Problem of Stability of a Stellar System with {Emden}'s Density Law and a Spherical Distribution of Velocities},
  booktitle = {Structure and Dynamics of Elliptical Galaxies},
  series    = {IAU Symposium},
  volume    = {127},
  pages     = {531--548},
  year      = {1987}
}

@article{Lebovitz,
  author  = {Lebovitz, Norman R.},
  title   = {On {Schwarzschild}'s Criterion for the Stability of Gaseous Masses},
  journal = {The Astrophysical Journal},
  volume  = {142},
  pages   = {229--242},
  year    = {1965}
}

@article{GU2016662,
  author  = {Gu, Xumin and Lei, Zhen},
  title   = {Local Well-Posedness of the Three Dimensional Compressible {Euler--Poisson} Equations with Physical Vacuum},
  journal = {Journal de Math\'{e}matiques Pures et Appliqu\'{e}es},
  volume  = {105},
  number  = {5},
  pages   = {662--723},
  year    = {2016}
}

@article{Xinzhouping,
  author  = {Luo, Tao and Xin, Zhouping and Zeng, Huihui},
  title   = {Well-Posedness for the Motion of Physical Vacuum of the Three-Dimensional Compressible {Euler} Equations with or without Self-Gravitation},
  journal = {Archive for Rational Mechanics and Analysis},
  volume  = {213},
  number  = {3},
  pages   = {763--831},
  year    = {2014}
}

@book{Hardy,
  author    = {Kufner, Alois and Maligranda, Lech and Persson, Lars-Erik},
  title     = {The Hardy Inequality: About Its History and Some Related Results},
  publisher = {Vydavatelsk\'{y} Servis},
  address   = {Plze\v{n}},
  year      = {2007}
}

@article{lin_2021_instability,
  author  = {Lin, Zhiwu and Zeng, Chongchun},
  title   = {Instability, Index Theorem, and Exponential Trichotomy for Linear Hamiltonian {PDEs}},
  journal = {Mem. Amer. Math. Soc.},
  volume  = {275},
  number  = {1347},
  pages   = {v+136},
  year    = {2022}
}
\end{document}